\documentclass[11pt,leqno]{amsart}
\usepackage{amscd}
\usepackage{color}
\usepackage[symbol]{footmisc}
\usepackage{amssymb}
\usepackage{amsfonts}
\usepackage{latexsym}
\usepackage{verbatim}
\usepackage{bbm}
\usepackage{mathrsfs}
 \usepackage[backref=page]{hyperref}
 \renewcommand*{\backref}[1]{}
\renewcommand*{\backrefalt}[4]{%
 	\ifcase #1 (Not cited).%
	\or        (Cited on page~#2).%
	\else      (Cited on pages~#2).%
	\fi}
\hypersetup{colorlinks=true,linkcolor=blue,citecolor=blue,urlcolor=black}
\newcommand{\N}{\mathbb{N}}

\newcommand{\R}{\mathbb{R}}
\newcommand{\C}{\mathbb{C}}
\newcommand{\f}{\varphi}
\renewcommand{\H}{\mathbb{H}}
\newcommand{\MA}{\mathrm{MA}}

\DeclareMathOperator{\QPSH}{\mathrm{PSH}^{\mathbb{H}}}
\renewcommand{\epsilon}{\varepsilon}
\renewcommand{\phi}{\varphi}
\renewcommand{\Re}{\mathrm{Re}}

\theoremstyle{plain}
\newtheorem{thm}{Theorem}[section]
\newtheorem{prop}[thm]{Proposition}
\newtheorem{lem}[thm]{Lemma}

\theoremstyle{plain}
\newtheorem{oss}[thm]{Remark}

\theoremstyle{definition}
\newtheorem{defn}[thm]{Definition}

\title[Continuous solutions to the quaternionic Monge--Amp\`ere equation]{Continuous solutions to the quaternionic Monge--Amp\`ere equation on locally flat HKT manifolds}

\begin{document}
 
\thanks{This work was supported by GNSAGA of INdAM, the first-named author was supported by the ANR--FAPESP project ANR-21-CE40-0017 {\it Bridges} and the DFF Sapere Aude grant \lq \lq Conformal geometry: metrics and cohomology''}

\address{(Giovanni Gentili) Aarhus University\\ Institut for Matematik\\
8000, Aarhus C\\ Denmark.}
\email{giovanni.gentili@math.au.dk}

\address{(Antonio Trusiani) Dipartimento di Matematica \\ Universit\`a di Roma Tor Vergata\\
00133 Roma\\ Italy.}
\email{trusiani@mat.uniroma2.it}

\address{(Luigi Vezzoni) Dipartimento di Matematica G. Peano \\ Universit\`a di Torino\\
Via Carlo Alberto 10\\
10123 Torino\\ Italy.}
\email{luigi.vezzoni@unito.it}

\author{Giovanni Gentili, Antonio Trusiani and Luigi Vezzoni}


\begin{abstract}
We prove that on a compact locally flat HKT manifold the quaternionic Monge--Amp\`ere equation always has a unique continuous solution. This provides new evidence for the Conjecture of Alesker and Verbitsky.
\end{abstract}

\maketitle

\section{Introduction}
Hyperk\"ahler manifolds with torsion (HKT manifolds) are special hypercomplex manifolds introduced by Howe and Papadopoulos in \cite{Howe-Papadopoulos (1996)} and then studied in several mathematical papers (see e.g.     
\cite{Alesker (2013),Alesker-Shelukhin (2017),Alesker-Verbitsky (2006),Alesker-Verbitsky (2010),Banos,Barberis-Fino,DinewSroka,dotti-fino2,GF,GentiliVezzoni,GLV,Grantcharov-Poon (2000),Ivanov,Lejmi-Weber,Sroka,Swann,Verbitsky (2002),Verbitsky (2007),Verbitsky (2009)} and the references therein).
A {\em hypercomplex manifold} is a smooth manifold $M$ equipped with a triple of complex structures $I,J,K$ satisfying the quaternionic-type relation 
\begin{equation}\label{IJK}
IJ=-JI=K\,.
\end{equation}
A hypercomplex manifold $(M,I,J,K)$ is HKT if it has a $2$-form $\Omega$ which is of type $(2,0)$ with respect to $I$ and satisfies  
\begin{enumerate}
\item[1.] the {\em q-reality condition} $\bar \Omega=J\Omega$; 

\vspace{0.1cm}
\item[2.] the {\em q-positivity condition} $\Omega(Z,J\bar Z)>0$ for every non-zero complex vector $Z$ which is of type $(1,0)$ with respect to $I$;

\vspace{0.1cm}
\item[3.] $\partial \Omega=0$, where  the operator $\partial$ is  computed with respect to $I$. 
\end{enumerate}
 
The HKT form $\Omega$ induces the Riemannian metric $g(X,Y)=\Re\, \Omega(X,JY)$ which is Hermitian with respect to each complex structure $I,J,K$.  

\medskip 
A central problem in HKT geometry is to seek \lq\lq special''  HKT structures in the  {\em quaternionic Bott--Chern cohomology class} of $\Omega$ 
$$
[\Omega]:=\{\Omega +\partial \partial_J\varphi\mid \varphi \in C^{\infty}(M)\}\,,
$$
where $\partial_J=J^{-1}\bar\partial J$, and $\partial$ and $\bar \partial$ are the usual Dolbeault operators with respect to $I$. A key ingredient in this direction is the following conjecture stated by Alesker and Verbitsky in \cite{Alesker-Verbitsky (2010)}, which is analogous to the Calabi conjecture \cite{Calabi,Yau} in the framework of quaternionic geometry:  

\medskip 
\noindent {\bf Conjecture (Alesker and Verbitsky \cite{Alesker-Verbitsky (2010)}).} {\em Let $(M^n,I,J,K,\Omega)$ be a compact HKT manifold such that $K_I(M)$ is holomorphically trivial. Then, for a given $F\in C^{\infty}(M)$ such that
\begin{equation*}\label{intF}
\int_M {\rm e}^F\,\Omega^n=\int_M \Omega^n \,,
\end{equation*}
the equation 
\begin{equation}\label{eq_QMA}
 (\Omega +\partial \partial_J \phi)^n={\rm e}^F\,\Omega^n\,,\quad \int_M \varphi \,\Omega^n=0 
\end{equation}
has a unique solution $\varphi\in C^{\infty}(M)$.} 

\medskip
In the statement of the conjecture $K_I(M)$ denotes the canonical bundle of $(M,I)$. We recall that the assumption on $K_I(M)$ to be holomorphically trivial implies the existence of a unique holomorphic volume form $\Theta$ on $(M,I)$ such that 
\begin{equation}\label{eq:normalisation}
\int_M\Omega^n \wedge \bar \Theta=1\,. 
\end{equation}
The form $\Theta$ always satisfies the q-positivity condition 
$$
\Theta(Z_1,JZ_1,Z_2,JZ_2,\dots, Z_{n},JZ_n)>0\mbox{ for every $(1,0)$-frame $\{Z_1,\dots,Z_n\}$ on $(M,I)$}\,,
$$
\cite[Thm 2.3]{Verbitsky (2007)}.

\medskip
\noindent In the statement of the conjecture and throughout the paper, we adopt the following notation:\\
\noindent \emph{Given a form $\Psi$ of type $(2n,0)$ on $(M,I)$, we write simply
$
\int_M \Psi
$
instead of $\int_M \Psi \wedge \overline{\Theta}.$}

\medskip 
One of the main geometric applications of the conjecture of Alesker and Verbitsky is the existence of Chern--Ricci flat HKT metrics on compact HKT  manifolds with holomorphically trivial canonical bundle. These metrics play the role in hypercomplex geometry that Calabi--Yau metrics play in K\"ahler geometry. 
 
Equation \eqref{eq_QMA} is called the \emph{quaternionic Monge--Amp\`ere equation} and it was originally considered by Alesker in open sets of the $n$-dimensional quaternionic space $\mathbb H^n$ \cite{Alesker}.
 
The conjecture of Alesker and Verbitsky is still open, but some partial results in the literature suggest that it could be true \cite{Alesker (2013),Alesker-Shelukhin (2017),Alesker-Verbitsky (2010),BGV,DinewSroka,GL,GentiliVezzoni,GV,GV2,Sro,Sroka,Sroka24}. In particular it is known that solutions to \eqref{eq_QMA} are unique \cite{Alesker-Verbitsky (2010)} and that the equation can be solved under some extra assumptions. In \cite{Alesker (2013)} Alesker confirmed the conjecture when the hypercomplex structure is locally flat and there is a hyperk\"ahler metric compatible with it. 

\medskip

A hypercomplex manifold is {\em locally flat} if it is locally isomorphic to open sets of $\mathbb H^n$. This condition was first considered by Sommese \cite{Sommese} who proved that it forces the underlying differential manifold to admit an atlas whose transition functions are affine quaternionic transformations. 
Moreover, the locally flat condition can be characterised in terms of the flatness of the Obata connection, i.e. the unique torsion free connection on a hypercomplex manifold preserving the hypercomplex structure \cite{Obata (1956)}. Typical examples of locally flat HKT manifolds are: flat hyperk\"ahler manifolds, Hopf hypercomplex manifolds \cite{Kato}, $2$-step nilmanifolds with left-invariant abelian hypercomplex structures (see \cite{dotti-fino2} and \cite[Prop. 6.1]{dotti-fino1}).  

\medskip

Recently Dinew and Sroka strongly improved Alesker's Theorem by removing the assumption of local flatness and confirming the conjecture on compact hyperk\"ahler manifolds \cite{DinewSroka}.

The next natural step in the study of the conjecture is the case when the manifold is locally flat, but does not admit any compatible hyperk\"ahler metric. This is the objective of the present paper where we prove that in the locally flat case equation \eqref{eq_QMA} can always be solved at least in a weak sense. This provides new evidence for the validity of the conjecture. 
Our methodology strictly follows the approach in the complex case developed in the paper of  Berman, Boucksom, Guedj and Zeriahi \cite{Berman-Boucksom-Guedj-Zeriahi} and is based on a variational technique via a quaternionic version of the Ding functional. An analogous method was implemented in several different scenarios. 

This approach requires a thorough investigation of the properties of {\em quaternionic plurisubharmonic functions}. The notion of quaternionic plurisubharmonicity was introduced by Alesker \cite{Alesker 2003}. Pluripotential theory in the quaternionic space was developed by Kang, Wan, Wang and Zhang through the study of quaternionic plurisubharmonic functions \cite{Wan 2017,Wan 2019,Wan-Kang,Wan-Wang,Wan-Zhang}. In particular, following some ideas of Bedford and Taylor \cite{Bedford-Taylor 1982}, Wan and Wang extended in \cite{Wan-Wang} the definition of the {\em quaternionic Monge--Amp\`{e}re operator} 
to locally bounded quaternionic plurisubharmonic functions on domains of $\H^n$. These definitions can be naturally further extended to compact locally flat HKT manifolds $(M^n,I,J,K,\Omega)$
by defining the set of $\Omega$\emph{-quaternionic plurisubharmonic functions} as the set of upper semicontinuous functions $\phi\in L^1(M,\R)$ such that $u+\varphi$ is quaternionic plurisubharmonic on any quaternionic chart $U$ where $\partial \partial_Ju=\Omega$ (see Section \ref{Sec 2}). Therefore the {\em quaternionic Monge--Ampère operator} 
\begin{equation}\label{MAoperator}
\MA_{\f}:= (\Omega+\partial \partial_J\f)^n
\end{equation}
extends to bounded $\Omega$-quaternionic plurisubharmonic functions. 
\medskip

The purpose of the present paper is to prove the following

\begin{thm}\label{thm:main}
Let $ (M^n,I,J,K,\Omega) $ be a compact locally flat HKT manifold such that $K_I(M)$ is holomorphically trivial. Then the quaternionic Monge--Ampère equation \eqref{eq_QMA} has always a unique continuous solution $\varphi$.
\end{thm}   

Theorem \ref{thm:main} is proved by developing a new approach to the Alesker--Verbitsky conjecture based on quaternionic pluripotential theory, rather than by adapting Yau's estimates from the classical Calabi conjecture. In view of \cite{Alesker (2013),DinewSroka}, the result is new when $(M,I,J,K)$ does not admit a compatible hyperk\"ahler metric. Natural examples to which the theorem applies are nilmanifolds \cite{dotti-fino1,dotti-fino2}.
Our approach could, hopefully, be extended to the non-locally flat case in future developments, in order to obtain a general version of Theorem \ref{thm:main} in the full framework of the Alesker--Verbitsky conjecture.

\smallskip

Theorem \ref{thm:main} is proved as follows: \\
In Section \ref{sec 5} we prove the existence of solutions in the space $\mathcal{E}^1(M,\Omega)$ of quaternionic plurisubharmonic functions with finite Monge--Ampère energy (we introduce this space in Section \ref{sec 4}).
Here, following the ideas in \cite{Berman-Boucksom-Guedj-Zeriahi}, we consider a Ding-type functional $\mathcal{F}\colon  \mathcal{E}^1(M,\Omega)\to \R$ and we show that it is an Euler--Lagrange functional of the quaternionic equation \eqref{eq_QMA} and that its critical points are weak solutions of the quaternionic Monge--Ampère equation. The strategy of the proof now adopts a variational method. We show that  $\mathcal{F}$ is proper, which allows us to deduce the existence of a maximiser in a finite energy class and thus the existence of a weak solution.

We then prove in Section \ref{Sec 6} that such a solution in $\mathcal{E}^1(M,\Omega)$ is necessarily bounded by adapting an argument in \cite{EGZ}. In Section \ref{Sec 7} following the ideas of \cite{K98} and \cite{Cho-Choi}, we show that any bounded solution in $\mathcal{E}^1(M,\Omega)$ is necessarily continuous. Finally, the uniqueness of $\varphi $ is achieved using suitable integration by parts in a similar way to the classical proof of uniqueness of solutions of the complex Monge--Ampère equations due to Calabi \cite{Calabi}.

In Section \ref{Sec 2}, we introduce quaternionic plurisubharmonic functions on locally flat HKT manifolds, proving some essential results such as a Hartogs' Compactness Theorem.
Section \ref{sec 3} is devoted to the definition of the quaternionic Monge–Ampère operator for unbounded $\Omega$-quaternionic plurisubharmonic functions. We also establish fundamental results in pluripotential theory, including the comparison principle and the maximum principle.

Finally, to improve the readability of the paper, we collect some technical results used in the paper in a final Appendix.

\medskip

\noindent {\bf Acknowledgements.} The authors are very grateful to Semyon Alesker, S\l awomir Dinew and Mehdi Lejmi for very useful comments on an early version of the paper.

\section{Quaternionic plurisubharmonic functions on locally flat HKT manifolds}\label{Sec 2}
{\em Quaternionic plurisubharmonic functions} on domains of $\H^n$ were introduced by Alesker in \cite{Alesker 2003} and their local theory was mostly developed in \cite{Alesker,Sro,Wan 2017,Wan 2019,Wan 2019b,Wan 2020,Wan-Kang,Wan-Wang,Wan-Zhang,Wang (2021)}. Throughout the paper by a {\em domain} in $\H^n$ we mean a connected open subset.

\begin{defn}
Given a domain $A\subseteq \H^n$, a function $ \phi \colon A \to [-\infty,\infty) $ is called {\em quaternionic plurisubharmonic} (quaternionic psh for short) if it is upper semicontinuous, $ \phi \not \equiv -\infty $ and it is either subharmonic or constant $ -\infty $ on each affine right quaternionic line, i.e. for every $ x\in A $ and $ v\in \H^n $ the function $ y\mapsto \phi(x+vy) $ is either subharmonic in the usual sense or constant $ -\infty $. We denote by $\QPSH(A)$ the set of quaternionic plurisubharmonic functions on $A$.
\end{defn}

Quaternionic plurisubharmonic functions are in particular subharmonic  \cite{Alesker 2003}, and for $ n=1 $ the two notions coincide. Therefore quaternionic psh functions inherit all the nice properties of subharmonic functions, for instance, they are in $ L^1_{\rm loc} $.
Even if, for many aspects, the theory of  quaternionic psh functions is analogous to the one of plurisubharmonic functions in the complex space, there are also some differences: for instance complex plurisubharmonic functions are always in $ L^p_{\mathrm{loc}} $ for any $ p\geq 1 $, while Sroka showed in \cite{Sro} that quaternionic psh functions are only in $ L^p_{\mathrm{loc}} $ for $ p<2 $ and such exponent is optimal. Moreover, another remarkable difference is that bounded plurisubharmonic functions in $ \C^n $ are necessarily constant, while bounded quaternionic psh functions on $ \H^n $ need not be. For example, for $ n=1 $ we know that quaternionic psh functions can be regarded as subharmonic functions in $ \R^4 $ and it is well-known that there exist bounded non-constant subharmonic functions in $ \R^4 $.

Quaternionic psh functions can be regularised via convolution (Proposition \ref{prop_convolution}), and $\QPSH(A)$ is a convex cone in $L^1_{\mathrm{loc}}(A)$ that is closed under decreasing limits (Proposition \ref{3}).  

\medskip 
A $C^2$-function $\varphi$ on an HKT manifold $(M,I,J,K,\Omega)$ is {\em $\Omega$-quaternionic plurisubharmonic}  if 
$\Omega+\partial \partial_J\f$ is q-positive. In the locally flat case the definition naturally extends to less regular functions by considering local quaternionic charts where $\Omega$ writes as $\partial \partial_Ju$ for a smooth function $u$ (such charts always exist \cite{Alesker-Verbitsky (2006),Banos}). We introduce the following

\begin{defn}
A function $ \phi \in L^1(M)$ on a locally flat HKT manifold $(M,I,J,K,\Omega)$ is called $\Omega $-{\em quaternionic plurisubharmonic} ($\Omega$-quaternionic psh for short) if in any local quaternionic chart such that $\Omega=\partial \partial_J u$ the function $ u+\phi $ is quaternionic plurisubharmonic. 
We denote by
\[
\QPSH(M,\Omega):= \{ \phi \in L^1(M) \mid \phi \text{ is $\Omega$-quaternionic psh} \}
\]
the space of $ \Omega $-quaternionic psh functions on $(M,I,J,K,\Omega)$. 
\end{defn}

Note that any $\Omega$-quaternionic psh function is bounded from above as it can be expressed locally as the difference of a quaternionic psh function and a local potential of $\Omega$, implying upper semicontinuity. The set $ \QPSH(M,\Omega) $ is naturally endowed with the $ L^1 $-topology. Moreover, as in the complex case (see, e.g., \cite[Prop. 8.2]{Guedj-Zeriahi book}), from the local theory it can be proved that the set $\QPSH(M,\Omega)$ is closed under maxima and convex combinations.

\medskip


The theory of $ \Omega $-quaternionic psh functions naturally interlaces with the study of currents. Given a compact HKT manifold $(M,I,J,K,\Omega)$ we denote by $ \mathcal{D}^{p,q}(M) $ the space of $ (p,q) $-currents on $M$ with respect to $I$. By definition $ \mathcal{D}^{p,q}(M) $
is the topological dual to $ \Lambda^{2n-p,2n-q}(M) $, where $ \Lambda^{2n-p,2n-q}(M) $ is with respect to $I$.  For instance, any $ (p,q) $-form $ \eta $ naturally defines a $ (p,q) $-current $ T_\eta $ by integration:
\[
T_\eta(\alpha)=\int_M \eta \wedge \alpha\,.
\]

The actions of $ I,J,K $ naturally extend to currents by duality; for instance the complex structure $ J\colon \mathcal{D}^{p,q}(M)\to \mathcal{D}^{q,p}(M) $ acts on $ T\in \mathcal{D}^{p,q}(M) $ as $
(JT)(\alpha)=T(J\alpha) $ for any $ \alpha \in \Lambda^{2n-p,2n-q}(M) $. Similarly, the operators $ \partial,\partial_J \colon \mathcal{D}^{p,q}\to \mathcal{D}^{p+1,q} $ extend to $ (p,q) $-currents by 
\begin{equation}\label{eqn:curr}
(\partial T)(\alpha)=(-1)^{p+q+1}T(\partial \alpha)\,, \qquad (\partial_J T)(\alpha)=(-1)^{p+q+1}T(\partial_J \alpha)\,.
\end{equation}
for any $ \alpha \in \Lambda^{2n-p,2n-q}(M) $. We recall here that a current $ T\in \mathcal{D}^{2p,2q} $ is called
\begin{itemize}
	\item \emph{q-real} if $ JT=\bar T $, where $ \bar T(\alpha):=\overline{T(\bar \alpha)} $ for any $ \alpha \in \Lambda^{2n-2p,2n-2q}(M) $;
	
	\vspace{0.1cm}
	\item \emph{q-positive} if it is q-real and additionally $ T(\alpha)\geq 0 $ for any q-positive $ \alpha \in \Lambda^{2n-2p,2n-2q}(M) $;
	
	\vspace{0.1cm}
	\item $ \partial $-\emph{closed} (resp. $ \partial_J $-\emph{closed}) if $ \partial T=0 $ (resp. $ \partial_JT=0 $);
	
	\vspace{0.1cm}
	\item \emph{q-closed} if it is both $ \partial $ and $ \partial_J$-closed.
\end{itemize}

For the second item we recall that a form $\alpha \in \Lambda^{2n-2p,2n-2q}(M)$ is called \emph{q-positive} if
\[
\alpha(Z_1,J\bar Z_1,\dots,Z_{n-p},J\bar Z_{n-p},\bar W_{1},JW_1,\dots,\bar W_{n-q},JW_{n-q})> 0
\]
for all non-vanishing linearly independent $Z_1,\dots,Z_{n-p}\in T^{1,0}_IM$ and all non-vanishing linearly independent $W_1,\dots,W_{n-q}\in T^{1,0}_IM$, equivalently, $\alpha$ lies in the interior of the cone generated by forms of the type
\[
\beta_1 \wedge J^{-1}\bar \beta_1 \wedge \dots \wedge  \beta_{n-p} \wedge J^{-1}\bar \beta_{n-p} \wedge \bar\gamma_1 \wedge J^{-1}\gamma_1\wedge \dots \wedge \bar \gamma_{n-q} \wedge J^{-1}\gamma_{n-q} \,,
\]
where $\beta_1,\dots,\beta_{n-p},\gamma_1,\dots,\gamma_{n-q} \in \Lambda^{1,0}_IM$ (see \cite{Verbitsky (2010)}).

Notice that a q-real current is $ \partial $-closed if and only if it is $ \partial_J $-closed. Moreover, by \cite[Prop. 3.4]{Wan-Wang}, any q-positive current can be regarded as a differential form whose coefficients are Radon measures. Recall that a  {\em Radon measure} is a Borel measure that is both inner and outer regular. 

\begin{oss}{\em 
The results in \cite{Wan-Kang,Wan-Wang,Wan-Zhang} are stated in terms of certain operators $ d_0, d_1, $ and their composition $ \Delta:=d_0d_1 $. Even if in general $d_0$ and $d_1$ differ from $\partial$ and $\partial_J$, it was proved by Sroka \cite[Prop. 1]{Sro} that after a suitable choice of coordinates, we have
\[
d_0=2\partial_J\,, \qquad d_1=-2\partial\,, \qquad \Delta=4\partial \partial_J\,.
\]
For this reason we will freely refer to the results contained in \cite{Wan-Kang,Wan-Wang,Wan-Zhang}, also when working with the operators $\partial$ and $\partial_J$.
}
\end{oss}

As a consequence of the next result, the definition of $\Omega$-quaternionic psh functions is independent of the choice of local potentials for $\Omega$, and the set $\QPSH(M,\Omega)$ essentially corresponds to the set of q-closed q-positive currents representing the (quaternionic Bott--Chern) cohomology class of $\Omega$.
\begin{prop}\label{prop:1}
Let $(M,I,J,K,\Omega)$ be a compact locally flat HKT manifold. If $\varphi\in \QPSH(M,\Omega)$, then $\Omega+\partial\partial_J \varphi$ is a q-closed q-positive current. Conversely if $\varphi\in L^1(M)$ is such that the q-closed current $\Omega +\partial \partial_J \varphi$ is q-positive, then there exists $\psi\in \QPSH(M,\Omega)$ such that $\phi=\psi$ almost everywhere. 
\end{prop}
\begin{proof}
Let $\varphi\in \QPSH(M,\Omega) $ and let $\alpha\in \Lambda^{2n-2,2n}(M)$ be a q-positive form. Consider a finite cover $\mathcal{U}:=\{U_j\}_{j=1}^N$ such that $\Omega_{|U_j}=\partial \partial_J u_j$ for local smooth potentials $u_j$, and denote by $\{\rho_j\}_{j=1}^N$ a partition of unity associated to $\mathcal{U}$. Then we have
$$
(\Omega+\partial\partial_J \phi)(\alpha)=\sum_{j=1}^N (\partial \partial_J (u_j+\varphi))(\rho_j \alpha)\geq 0\,,
$$
where the last inequality follows from \cite[Prop. 3.7]{Wan-Wang}. 

Conversely, let $\varphi\in L^1(M)$ be such that $\Omega+\partial \partial_J\varphi\geq 0$. It follows that
$$
0\leq (\Omega+\partial\partial_J \phi)(\alpha)=(\partial \partial_J (u+\phi))(\alpha)\,,
$$
for any q-positive form $\alpha\in \Lambda^{2n-2,2n}(M)$ that has compact support contained in a local quaternionic chart $U$, where $\Omega\vert_U=\partial\partial_J u_U$.
Take a local standard regularisation $ \psi_{U,\epsilon}=(u_U+\phi) \star \chi_\epsilon $. By \cite[Prop. 2.1.6]{Alesker 2003} we see that $ \partial \partial_J \psi_{U,\epsilon}=(\partial \partial_J (u_U+\phi))\star \chi_\epsilon\geq 0 $ on $U_\epsilon:=\{ x\in U \mid \mathrm{dist}(x,\partial U) > \epsilon\}$ as $ \psi_{U,\epsilon} \in \QPSH(U_\epsilon)\cap C^{\infty}(U_\epsilon)$ (Proposition \ref{prop_convolution}). Hence, for $ \epsilon \searrow 0 $ the regularisation $ \psi_{U,\epsilon} $ decreases to a quaternionic psh function $ \psi_U $ (Proposition \ref{3}), but since $ \psi_{U,\epsilon} \to u_U+\phi $ in $ L^1_{\rm loc}(U) $ we must have $ \psi_U=u_U+\phi $ almost everywhere on $U$.
Next, we observe that for any two quaternionic charts $U_1,U_2$ such that $U_1\cap U_2\neq \emptyset$, we have
$$
\psi_{U_1}=\psi_{U_2}+(u_{U_1}-u_{U_2})
$$
almost everywhere on $U_1\cap U_2$. But $\psi_{U_1}, \psi_{U_2}+(u_{U_1}-u_{U_2})\in \QPSH(U_1\cap U_2)$, thus they are equal on $U_1\cap U_2$ (this immediately follows from Proposition \ref{prop_convolution} as their local regularisations coincide). Hence the functions $\psi_{U_j}-u_{U_j}$ glue together to a global function $\psi$ that is equal to $\phi$ almost everywhere and belongs to $\QPSH(M,\Omega)$.
\end{proof}

We end this section with the following compactness result which will be fundamental in the proof of the main theorem.

\begin{lem}[Hartogs' Compactness Theorem]\label{lem:rel_cptness}
Let $(M,I,J,K,\Omega)$ be a compact locally flat HKT manifold.	The set
\[
\left\{ \phi \in \QPSH(M,\Omega) \mid  \sup_M \phi=0 \right\}
\]
is compact in $ \QPSH(M,\Omega) $. Furthermore, let $\mu$ be a positive Radon measure that satisfies $\QPSH(M,\Omega) \subseteq L^1(\mu)$, then the subset $ \{ \phi \in \QPSH(M,\Omega) \mid \int_M \phi\, d\mu=0 \} $ is relatively compact. In particular there exists $C$ such that for any $ \phi \in \QPSH(M,\Omega)$,
	\begin{equation}
	    \label{eqn:L1_sup_comparison}
        -C+ \mu(M)\sup_M \phi \leq \int_M \phi\, d\mu \leq \mu(M)\sup_M \phi\,.
	\end{equation}
\end{lem}
\begin{proof}
Set $K:=\left\{\phi\in \QPSH(M,\Omega)\, |\, \sup_M \phi=0\right\}$. Letting $(\phi_j)_{j\in \N}\subset K$, Proposition \ref{prop_basic}.(4) yields a subsequence $(\phi_{j_k})_{k\in \N}\subset K$ that converges to a function $\phi\in \QPSH(M,\Omega)$. Then, as immediate consequence of Proposition \ref{prop_basic}.(3), we have that necessarily $\sup_M \phi=0$, i.e. $\phi\in K$ and $K$ is compact.

The right inequality in \eqref{eqn:L1_sup_comparison} is trivial. Let now assume by contradiction to have a sequence $(\phi_j)_{j\in \N}$ in  $\QPSH(M,\Omega)$ such that $ \int_M \phi_j\, d\mu=0 $ and such that $\mu(M)\sup_M \phi_j\geq 2^j$ for all $j\in\N$. Setting $\psi_j:=\phi_j-\sup_M \phi_j $, by the first part of the proof, up to passing to a subsequence, $ \psi_{j_k} \to \psi \in L^1(M) $. We observe that 
\[
\|\psi_j\|_{L^1(\mu)}=\int_M \sup_M \phi_j\, d\mu-\int_M \phi_j\, d\mu =\mu(M)\sup_M \phi_j\geq 2^j\,.
\]
To derive the contradiction it then remains to prove that $\| \psi_j \|_{L^1(\mu)}$ is uniformly bounded in $j$.

Assume first that $\mu$ is smooth. Then $\psi_j \mu \to \psi \mu$ in the weak sense of measures. Consequently $\int_M \psi_j\, d \mu \to \int_M \psi\, d \mu>-\infty$, implying that $ \|\psi_j\|_{L^1(\mu)} $ is uniformly bounded.

If $\mu$ is not smooth, 
set $\psi=\sum_{j = 1}^\infty 2^{-j} \psi_j$. As the $ L^1 $ norm of $ \psi_j $ with respect to a smooth positive Radon measure is uniformly bounded, we have $ \psi \not \equiv -\infty $. Moreover $\psi\in \QPSH(M,\Omega)$, being decreasing limit of a sequence of functions in $ \QPSH(M,\Omega) $ (see Proposition \ref{prop_basic}.(1)). In particular $\psi\in L^1(\mu)$ and we get a contradiction observing that by the monotone convergence theorem we have $\int_M \psi\, d\mu=\sum_{j = 1 }^\infty 2^{-j} \int_M \psi_j\, d\mu=-\infty$.
\end{proof}




\section{The quaternionic Monge--Amp\`{e}re operator}\label{sec 3}
Let $ (M,I,J,K,\Omega) $ be a compact locally flat HKT manifold. 

\medskip
The {\em quaternionic Monge--Amp\`ere operator} \eqref{MAoperator}
naturally defined for $C^2$-functions, can be extended to functions in $\QPSH(M,\Omega)\cap L^\infty(M)$ by using currents. Although in general the wedge product of currents is not defined, if $\varphi\in \QPSH(M,\Omega)\cap L^\infty(M)$ and $k\in \mathbb N$, we can define $ (\Omega+\partial \partial_J \phi)^k $ as follows:

Firstly, given a current $ T\in \mathcal{D}^{p,q}(M) $ and $ \eta \in \Lambda^{r,s}(M) $ on $(M,I)$ their wedge product is defined as the $(p+r,q+s)$-current acting on  $ \alpha \in \Lambda^{2n-p-r,2n-q-s}(M) $ as 
\[
(T\wedge \eta)(\alpha):=T(\eta \wedge \alpha)\,;
\]
then for $\phi \in \QPSH(M,\Omega)\cap L^\infty(M) $ and a q-closed q-positive $ T\in \mathcal{D}^{2p,0}(M) $, the product $ \phi T $ is well-defined and, consequently, so is the current $ \partial \partial_J \phi \wedge T$ via the relation
\[
\partial \partial_J \phi \wedge T:=\partial \partial_J(\phi T)\,.
\]
The current $ \partial \partial_J \phi \wedge T$ is q-closed and, local regularisation (Proposition \ref{prop_convolution}) together with Proposition \ref{prop:1}, implies that $(\Omega+ \partial \partial_J \phi) \wedge T$ is also q-positive. Proceeding inductively one defines $ (\Omega+\partial \partial_J \phi)^k $. We can then regard $\MA$ as an operator $\QPSH(M,\Omega)\cap L^\infty(M)\to \mathcal{D}^{2n,0}(M)$. More generally, given $\phi_1,\dots,\phi_n\in \QPSH(M,\Omega)\cap L^\infty(M)$, we define the {\em quaternionic mixed Monge--Amp\`ere operator}
\[
\MA_{\phi_1,\dots,\phi_n}:= (\Omega+\partial \partial_J \phi_1) \wedge \dots \wedge (\Omega+\partial \partial_J \phi_n)\,,
\]
which is clearly symmetric in the $\phi_i$'s.

We have the following 

\begin{thm}\label{thm:cont}
	If $ \phi^0_j,\dots,\phi^k_j $ are sequences in $ \QPSH(M,\Omega)\cap L^\infty(M) $, such that either one of the following holds:
	\begin{itemize}
		\item $ \phi^i_j $ decreases pointwise to a $\Omega$-quaternionic plurisubharmonic function $ \phi^i\in L^\infty(M) $ for each $ i=0,\dots,k $;
		\item $ \phi^i_j $ is locally uniformly bounded and increases almost everywhere
		to a $\Omega$-quaternionic plurisubharmonic function $ \phi^i\in L^\infty(M) $ for each $ i=0,\dots,k $;
	\end{itemize}
	 then
	\[
	\phi^0_j\partial \partial_J \phi^1_j\wedge \dots \wedge \partial \partial_J \phi^k_j \to \phi^0\partial \partial_J \phi^1\wedge \dots \wedge \partial \partial_J \phi^k\,,
	\]
	in the weak sense of currents.
\end{thm}
\begin{proof}
Following \cite{Bedford-Taylor 1976}, such continuity can be deduced directly from the local theory as the manifold is assumed to be locally flat, cf. \cite[Prop. 3.2]{Wan-Zhang}.
\end{proof}

From now on we assume that $K_I(M)$ is holomorphically trivial and we adopt the convention described in the introduction by omitting to write the form $\Theta$ inside the integrals.  
Observe that the holomorphicity of $\Theta$ implies $\int_M \partial T=0$ for any $(2n-1,0)$-current $T$. We prove two key results in pluripotential theory: the maximum and the comparison principle. Given a subset $U$ of $M$, we denote by ${\mathbf 1}_{U}$ the characteristic function of $U$. 


\begin{prop}\label{Prop:max_princ}
Let $ \phi_1,\dots,\phi_n,\psi\in \QPSH(M,\Omega)\cap L^\infty(M) $, then
$$
\mathbf{1}_{\{\phi_1>\psi\}}\MA_{\phi_1,\phi_2,\dots,\phi_n}=\mathbf{1}_{\{\phi_1>\psi\}}\MA_{\max(\phi_1,\psi),\phi_2,\dots,\phi_n}.
$$
In particular, for every $\phi,\psi \in \QPSH(M,\Omega) \cap L^\infty(M)$
$$
\mathbf{1}_{\{\phi>\psi\}}\MA_{\phi}=\mathbf{1}_{\{\phi>\psi\}}\MA_{\max(\phi,\psi)}\,.
$$
\end{prop}
\begin{proof}
	Since the nature of the result is local and the manifold is locally flat, it is enough to prove the statement on a domain $ A\subseteq \H^n $. Namely we want to show that
	\[
	\mathbf{1}_{\{u>v\}}(\partial \partial_J u)\wedge T=\mathbf{1}_{\{u>v\}}(\partial \partial_J \max(u,v))\wedge T
	\]
	as Borel measures, where $ u,v\in \QPSH(A)\cap L^\infty_{\mathrm{loc}}(A) $ and $ T $ is a q-closed q-positive $ (2n-2,0) $-current.
	
	If $ u $ is continuous, $ \{u >v \} $ is an open subset of $ A $ and the result follows. If $ u $ is not continuous we take a sequence $ (u_j)_{j\in \N}\subset \QPSH(A_\epsilon)\cap C^\infty(A_\epsilon) $ decreasing to $ u $ (Proposition \ref{prop_convolution}). We have
	\[
	\mathbf{1}_{\{u_j>v\}}(\partial \partial_J \max(u_j,v))\wedge T=	\mathbf{1}_{\{u_j>v\}}(\partial \partial_J u_j)\wedge T\,.
	\]
	Let $ w_j=\sup(u_j-v,0) $ and $ w=\sup(u-v,0) $ and observe that $ (w_j)_{j\in \N} $ decreases to $ w $. Therefore, by continuity of the Monge--Amp\`{e}re operator (cf. \cite[Prop. 3.2]{Wan-Zhang})
	\[
	w(\partial \partial_J \max(u,v))\wedge T=\lim_{j\to +\infty} w_j(\partial \partial_J \max(u_j,v))\wedge T=	\lim_{j\to +\infty} w_j (\partial \partial_J u_j)\wedge T=	w(\partial \partial_J u)\wedge T
	\]
	in the sense of Borel measures. Since $ 1/(w+\epsilon) $ is bounded for every $ \epsilon>0 $ we have
	\[
	\frac{w}{w+\epsilon}(\partial \partial_J \max(u,v))\wedge T = \frac{w}{w+\epsilon}(\partial \partial_J u)\wedge T\,,
	\]
	and this allows to conclude letting $ \epsilon $ decrease to $0 $ because $ w/(w+\epsilon) $ increases to $ \mathbf{1}_{\{u>v\}} $.
%
\end{proof}

\begin{prop}\label{Prop:cps_principle}
	Let $ \phi_1,\dots,\phi_n,\psi \in \QPSH(M,\Omega)\cap L^\infty(M) $, then
	\[
	\int_{\{\phi_1>\psi\}} \MA_{\phi_1,\phi_2,\dots,\phi_n} \leq \int_{\{\phi_1>\psi\}} \MA_{\psi,\phi_2,\dots,\phi_n}\,.
	\]
    In particular
    \begin{equation}\label{eq:MAmass}
    \MA_{\phi_1,\dots,\phi_n}(M)=\int_M \Omega^n=1
    \end{equation}
    for any $\phi_1,\dots,\phi_n \in \QPSH(M,\Omega)\cap L^\infty(M)$.
\end{prop}
\begin{proof}
	From Proposition \ref{Prop:max_princ} we have
    \[
    \begin{split}
	\int_{\{\phi_1>\psi\}} \MA_{\phi_1,\phi_2,\dots,\phi_n}&=\int_{\{\phi_1>\psi\}} \MA_{\max(\phi_1,\psi),\phi_2,\dots,\phi_n}=1-\int_{\{\phi_1\leq \psi\}} \MA_{\max(\phi_1,\psi),\phi_2,\dots,\phi_n}\\
    &\leq 1-\int_{\{\phi_1< \psi\}} \MA_{\psi,\phi_2,\dots,\phi_n}=\int_{\{\phi_1\geq \psi\}} \MA_{\psi,\phi_2,\dots,\phi_n}\,.
    \end{split}
    \]
	The desired inequality follows by replacing $ \psi $ with $ \psi+\epsilon $, so that
	\[
	\int_{\{\phi_1>\psi+\epsilon \}} \MA_{\phi_1,\phi_2,\dots,\phi_n}\leq \int_{\{\phi_1\geq \psi+\epsilon \}} \MA_{\psi,\phi_2,\dots,\phi_n} \leq \int_{\{\phi_1> \psi\}} \MA_{\psi,\phi_2,\dots,\phi_n}\,,
	\]
	and letting $ \epsilon \searrow 0 $.
    The equality \eqref{eq:MAmass} follows by an easy calculation together with \eqref{eq:normalisation}.
\end{proof}

The definition of the Monge--Amp\`{e}re operator can be extended to unbounded $ \Omega $-quaternionic psh functions as follows.

\begin{defn}
For $ \phi \in \QPSH(M,\Omega)$, we define the \emph{Monge--Ampère measure of $\phi$} as
$$
\MA_\phi:=\lim_{j\to +\infty} \mathbf{1}_{\{\phi>-j\}}\MA_{\max(\phi,-j)}\,.
$$
More generally, given $\phi_1,\dots,\phi_n\in \QPSH(M,\Omega)$, we define the \emph{quaternionic mixed Monge--Ampère measure of $\phi_1,\dots,\phi_n$} as
    $$
    \MA_{\phi_1,\dots, \phi_n}:=\lim_{j\to +\infty}\mathbf{1}_{\cap_{s=1}^n\{\phi_s>-j\}}\MA_{\max(\phi_1,-j),\dots,\max(\phi_n,-j)}\,.
    $$
The functions $ \max (\phi,-j)  \in \QPSH(M,\Omega)\cap L^\infty(M)$ are called the {\em canonical approximants} of $\varphi$.
\end{defn}
The next lemma shows that $\MA_\phi, \MA_{\phi_1,\dots,\phi_n}$ are well-defined positive Borel measures.


\begin{lem}\label{Lem:MAdef}
Let $\phi_1,\dots,\phi_n\in \QPSH(M,\Omega)$, then the sequence of positive Borel measures $ \mu_j:=\mathbf{1}_{\cap_{s=1}^n\{\phi_s>-j\}}\MA_{\max(\phi_1,-j),\dots, \max(\phi_n,-j)} $ is non-decreasing and $\mu_j(M)\leq 1$. In particular, $\MA_{\phi_1,\dots, \phi_n}$ is a positive Borel measure such that $ \MA_{\phi_1,\dots,\phi_n}(M)\leq 1$.
\end{lem}
\begin{proof}
We prove the result when $\phi_s=\phi$ for all $s=1,\dots,n$ as the general case follows by a simple adaptation of the proof. Denoting then by $\phi_j:=\max(\phi,-j)$ the canonical approximants of $\phi$, Proposition \ref{Prop:max_princ} implies
\[
{\mathbf 1}_{\{\phi_j>-k\}}\MA_{\phi_j}={\mathbf 1}_{\{\phi_j>-k\}}\MA_{\max(\phi_j,-k)}\,;
\]
hence, for $ j\geq k $ we obtain
\[
{\mathbf 1}_{\{\phi>-j\}}\MA_{\phi_j}
\geq
{\mathbf 1}_{\{\phi>-k\}}\MA_{\phi_j}={\mathbf 1}_{\{\phi>-k\}}\MA_{\phi_k}\,.
\]
This shows that the sequence $ \mu_j:=\mathbf{1}_{\{\phi>-j\}}\MA_{\phi_j} $ is non decreasing. As $\mu_j\leq \MA_{\phi_j}$, by \eqref{eq:MAmass} we get $\mu_j(M)\leq 1$.
Thus, 
\[
\MA_\phi:=\lim_{j\to +\infty }\mu_j=\lim_{j\to +\infty} \mathbf{1}_{\{\phi>-j\}}\MA_{\phi_j}
\]
is itself a positive Borel measure with total mass bounded from above by $ 1 $.
\end{proof}

For later use we prove the following generalisation of Proposition \ref{Prop:max_princ}.

\begin{lem}\label{lem:max_princ_Unbounded}
    Let $\phi_1,\dots,\phi_n\in\QPSH(M,\Omega)$ and $\psi\in\QPSH(M,\Omega)\cap L^\infty(M)$. Then 
    $$
    \mathbf{1}_{\{\phi_1>\psi\}}\MA_{\phi_1,\phi_2,\dots,\phi_n}=\mathbf{1}_{\{\phi_1>\psi\}}\MA_{\max(\phi_1,\psi),\phi_2,\dots,\phi_n}.
    $$
\end{lem}
\begin{proof}
By definition of $\MA_{\phi_1,\dots,\phi_n}$ we have
$$
\mathbf{1}_{\{\phi_1>\psi\}}\MA_{\phi_1,\phi_2,\dots,\phi_n}= \lim_{j\to +\infty}\mathbf{1}_{ \{\phi_1>\psi\}}\mathbf{1}_{\cap_{s=1}^n\{\phi_s>-j\}}\MA_{\max(\phi_1,-j),\dots,\max(\phi_n,-j)}\,.
$$
We then observe that
$$
\{\phi_1>\psi\}\cap \{\phi_1>-j\}=\{\max(\phi_1,-j)>\psi\}\cap \{\phi_1>-j\}=\{w>-j\}\cap \{\phi_1>\psi\}\,,
$$
where we set $w:=\max(\phi_1,\psi)$. Set also $\varphi_2^j:=\max(\varphi_2,-j), \dots,\varphi_n^j:=\max(\varphi_n,-j)$ to lighten notations. It follows from Proposition \ref{Prop:max_princ} that
\[
\begin{split}
    \mathbf{1}_{\{\phi_1>\psi\}}\MA_{\phi_1,\phi_2,\dots,\phi_n}
    &=\lim_{j\to +\infty}\mathbf{1}_{\{\max(\phi_1,-j)>\psi\}}\mathbf{1}_{\cap_{s=1}^n \{\phi_s>-j\}}\MA_{\max(\phi_1,-j),\phi_2^j,\dots,\phi_n^j}\\
    &=\lim_{j\to +\infty}\mathbf{1}_{\{\max(\phi_1,-j)>\psi\}}\mathbf{1}_{\cap_{s=1}^n \{\phi_s>-j\}}\MA_{\max(\max(\phi_1,-j),\psi),\phi_2^j,\dots,\phi_n^j}\\
    &=\lim_{j\to +\infty}\mathbf{1}_{\{w>-j\}\cap\{\phi_1>\psi\}}\mathbf{1}_{\cap_{s=2}^n \{\phi_s>-j\}}\MA_{\max(w,-j),\phi_2^j,\dots,\phi_n^j}\\
    &=\mathbf{1}_{\{\phi_1>\psi\}}\MA_{w,\phi_2,\dots,\phi_n}\,,
\end{split}
\]
which concludes the proof.
\end{proof}

A subset of $M$ is called {\em $Q$-polar} if it is contained in the $-\infty$-locus of an $\Omega$-quaternionic psh function (cf. \cite{Wan-Kang}).

\begin{prop}\label{prop:NonPluripolar}
    Let $\phi_1,\dots,\phi_n \in \QPSH(M,\Omega)$, then the mixed Monge--Ampère operator $\MA_{\phi_1,\dots,\phi_n}$ does not put mass on $Q$-polar sets, i.e. $ \MA_{\phi_1,\dots,\phi_n}(E)=0 $ for all $Q$-polar sets $E\subseteq M$.
\end{prop}
\begin{proof}
By definition of the quaternionic mixed Monge--Ampère measure $\MA_{\phi_1,\dots,\phi_n}$ it is sufficient to treat the case $\phi_i\in \QPSH(M,\Omega)\cap L^\infty(M)$ (cf. Lemma \ref{Lem:MAdef}). Moreover, as a consequence of the quaternionic version of Josefson's Theorem, see \cite[Thm. 1.1]{Wan-Kang}, it is enough to work locally on a domain $A\subset \H^n$. Then thanks to the Chern-Levine-Nirenberg inequalities proved in \cite[Prop. 6.3]{Alesker 2012} and an adaptation of the proof in the complex case (cf. \cite[Lem. 2.2]{Wan-Kang}) we obtain that for any subdomain $A'\Subset A$ and any compact set $K\subset A'$ there exists a constant $C=C(A',K)>0 $ such that
\begin{equation}
    \label{eqn:CLN}
    \int_K \lvert v\rvert \partial \partial_J u_1 \wedge \cdots \wedge \partial \partial_J u_n\leq C \lVert v \rVert_{L^1(A')} \lVert u_1 \rVert_{L^\infty(A')}\cdots \lVert u_n\rVert_{L^\infty(A')}
\end{equation}
for any $v\in \QPSH(A), u_i\in \QPSH(A)\cap C^2(A)$. Combining Proposition \ref{prop_convolution} with the local version of Theorem \ref{thm:cont} (cf. \cite[Prop. 3.2]{Wan-Zhang}) we can extend \eqref{eqn:CLN} to the case $u_i\in\QPSH(A)\cap L^\infty(A)$. In particular the quaternionic mixed Monge--Ampère measure $\partial \partial_J u_1\wedge \cdots \wedge \partial \partial_J u_n$ puts no mass on $Q$-polar sets when $v\in \QPSH(A)$ and $u_i\in\QPSH(A)\cap L^\infty(A)$, which concludes the proof.
\end{proof}



\section{The finite energy class \texorpdfstring{$\mathcal{E}^1(M,\Omega)$}{E}}\label{sec 4}
Let $ (M,I,J,K,\Omega) $ be a compact locally flat HKT manifold with $K_I(M)$ holomorphically trivial.

For later use we define $\partial \phi \wedge \partial_J \phi\wedge T$ for a given $\varphi\in \QPSH(M,\Omega)\cap L^\infty(M)$ and a q-closed q-positive current $T$, as 
\begin{align*}
&\partial \phi \wedge \partial_J \phi\wedge T:=\frac{1}{2}\partial \partial_J(\phi^2)\wedge T-\phi \partial \partial_J \phi \wedge T\,.
\end{align*}
\begin{lem}\label{lem:fakegradient}
    Let $T$ be a q-closed q-positive $(2n-2,0)$-current. If $\varphi$ is a linear combination of bounded $\Omega$-quaternionic psh functions, then
    $$
    \int_M\partial \varphi \wedge \partial_J \varphi\wedge T\geq 0.
    $$ 
\end{lem}
\begin{proof}
   Write $\phi=\sum_j\lambda_j \phi_j$, for $\lambda_j\in \R$ and $\phi_j\in \QPSH(M,\Omega)\cap L^\infty(M)$. Proposition \ref{Prop:approx} implies that such a $\varphi$ can be approximated with a family of smooth functions $\{\varphi_\epsilon\}_{\varepsilon>0}$ which are linear combinations of quaternionic psh functions. It follows from Theorem \ref{thm:cont} that 

\[
\int_M \partial \phi \wedge \partial_J \phi  \wedge T=-
\int_M \phi\partial \partial_J \phi\wedge T=-\lim_{\varepsilon\searrow 0} \int_M \phi_\varepsilon \partial \partial_J \phi_\varepsilon \wedge T=\lim_{\varepsilon\searrow 0} \int_M \partial \phi_\epsilon \wedge \partial_J \phi_\epsilon  \wedge T \geq 0\,,
\]
where in the last step we have used that $T$ is a q-positive current and that the $(2,0)$-form $\partial \phi_\epsilon \wedge \partial_J \phi_\epsilon$ regarded as a current is q-positive as well. 
\end{proof}

We introduce the following classes of $\Omega$-quaternionic plurisubharmonic functions:
\[
\begin{aligned}
\mathcal{E}(M,\Omega)&:=\left\{ \phi \in \QPSH(M,\Omega) \mid \int_M \MA_{\varphi}=1
\right\}\,,\\
\mathcal{E}^1(M,\Omega)&:=\{ \phi \in \mathcal{E}(M,\Omega) \mid \phi \in L^1(\MA_\phi) \}\,. 
\end{aligned}
\]
The class $ \mathcal{E}^1(M,\Omega) $ is invariant by translations, i.e. $ \phi \in \mathcal{E}^1(M,\Omega) $ if and only if $ \phi+c\in \mathcal{E}^1(M,\Omega) $ for any constant $ c\in \R $.
We also define the \emph{quaternionic Monge--Amp\`{e}re energy functional} $ E \colon \QPSH(M,\Omega)\cap L^\infty(M)\to \R $ as 
\begin{equation}\label{eqn:energy}
E(\phi):=\frac{1}{n+1
}\sum_{j=0}^n\int_M \phi\, (\Omega+\partial \partial_J\varphi)^j \wedge \Omega^{n-j} \,.
\end{equation}
The definition of the energy extends to $ \QPSH(M,\Omega) $ by setting
\[
E(\phi):=\inf \{ E(\psi) \mid \phi \leq \psi\in \QPSH(M,\Omega)\cap L^\infty(M) \}\,,
\]
this is coherent with the following monotonicity of $ E $.

\begin{prop}\label{prop_Echi_incr}
The energy \eqref{eqn:energy} is non-decreasing and concave. Furthermore for any non-positive $ \phi \in \QPSH(M,\Omega)\cap L^\infty(M) $ we have 
\[
\int_M \phi\, \MA_\phi \leq E(\phi) \leq \frac{1}{n+1
} \int_M \phi\, \MA_\phi\,.
\]
Moreover, $ E $ is upper semicontinuous in the $ L^1 $-topology and it is continuous along decreasing sequences.
\end{prop}
\begin{proof}
Let $ \phi,\psi \in \QPSH(M,\Omega)\cap L^\infty(M) $ and assume $ \phi\leq \psi $. Set $ \phi_t:=(1-t)\phi+t\psi\in \QPSH(M,\Omega)\cap L^\infty(M) $. From a straightforward computation and Lemma \ref{lem:fakegradient} we obtain
\begin{equation}\label{eq_primitive}
	\frac{d}{dt} E(\phi_t)=
	\int_M \dot \phi _t\, \MA_{\phi_t}\geq 0\,,
\end{equation}
\[
	\frac{d^2}{dt^2}E(\phi_t)=-n\int_M \partial\dot \phi_t \wedge \partial_J\dot \phi_t \wedge (\Omega+\partial \partial_J \phi_t)^{n-1}\leq 0\,,
\]
showing that the energy is non-decreasing and concave.

Let $ \phi \in \QPSH(M,\Omega)\cap L^\infty(M) $ be non-positive, then it is clear that $ E(\phi) \leq \frac{1}{n+1
}\int_M \phi \MA_\phi $. The other inequality is implied by the following
\begin{equation}\label{eq:trick}
\begin{split}
	\int_M \phi\, \Omega_\phi^{j+1}\wedge \Omega^{n-j-1}-\int_M \phi\,\Omega_\phi^{j}\wedge \Omega^{n-j}& =\int_M \phi\partial \partial_J \phi \wedge \Omega_\phi^{j}\wedge \Omega^{n-j-1}\\
	&=- \int_M \partial\phi \wedge \partial_J \phi \wedge \Omega_\phi^{j}\wedge \Omega^{n-j-1}\leq 0\,,
\end{split}
\end{equation}
where we wrote $\Omega_\phi:=\Omega+ \partial \partial_J \phi$ and the last inequality is again given by Lemma \ref{lem:fakegradient}.
It remains to prove upper semicontinuity and continuity along decreasing sequences. Assume first that $(\varphi_j)_{j\in \N}$ decreases to $\varphi\in \QPSH(M,\Omega)\cap L^\infty(M)$. From \eqref{eq_primitive} a simple calculation gives
$$
0\leq E(\varphi_j)-E(\varphi)=\frac{1}{n+1}\sum_{k=0}^n \int_M (\varphi_j-\varphi)\left(\Omega+\partial \partial_J \varphi_j\right)^k\wedge\left(\Omega+\partial \partial_J \varphi\right)^{n-k}\leq \int_M (\varphi_j-\varphi)\MA_{\varphi}\,,
$$
where the last inequality follows from an integration by parts argument together with Lemma \ref{lem:fakegradient}, similarly to \eqref{eq:trick}.
By the Dominated Convergence Theorem it follows that $E(\varphi_j)\to E(\varphi)$.

Let us now consider $\varphi\in \QPSH(M,\Omega) $ and set $\varphi^k:=\max(\varphi,-k)$ for the canonical approximants. For any $u\in\QPSH(M,\Omega)\cap L^\infty(M)$ such that $u\geq \varphi$ there exists $k\in\N$ such that $u\geq \varphi^k$. In particular $E(u)\geq \lim_k E(\varphi^k)$, from which we deduce that $E(\varphi^k)\to E(\varphi)$ taking the infimum over all the $u$'s and using the monotonicity of the Monge--Ampère energy.

Next take a sequence $(\varphi_j)_{j\in \N}\subset \QPSH(M,\Omega)$ such that $\varphi_j\to \varphi\in\QPSH(M,\Omega)$ in $L^1$ and set $\varphi_j^k:=\max(\varphi_j,-k), \varphi^k:=\max(\varphi, -k)$. The sequence $\psi_j^k:=(\sup_{l\geq j}\varphi_l^k)^*$ belongs to $\QPSH(M,\Omega)$ and decreases to $\varphi^k$ as $j\to +\infty$ for any $k\in\N $ fixed (Proposition \ref{prop_basic}.(3)). As $\varphi^k$ is bounded, we obtain
$$
\lim_j E(\psi^k_j)= E(\varphi^k)
$$
by what was proved above. By monotonicity of the Monge--Ampère energy we deduce
$$
\limsup_j E(\varphi_j)\leq \lim_j E(\psi_j^k)=E(\varphi^k)
$$
for any $k\in\N$. Letting $k\to +\infty$ yields $E(\varphi^k)\to E(\varphi)$, which concludes the proof of the upper semicontinuity. If moreover $\varphi_j\searrow \varphi$, again by monotonicity of the energy functional we also have $ E(\phi)\leq \liminf_{j\to +\infty} E(\phi_j) $, thus continuity of $E$ follows. 
\end{proof}

From the previous proposition it follows 
\[
\mathcal{E}^1(M,\Omega)=\{ \phi \in \mathcal{E}(M,\Omega) \mid E(\phi)>-\infty \}\,.
\]
\begin{prop}\label{prop:cptness}
    For any $C>0$ the set 
    $$
    \mathcal{E}^1_C(M,\Omega):=\{ \phi \in \mathcal{E}^1(M,\Omega) \mid E(\phi)\geq -C\,, \, \phi \leq 0 \}\subseteq \mathcal{E}^1(M,\Omega)
    $$ 
    is convex and compact with respect to the $L^1$-topology.
\end{prop}
\begin{proof}
    The convexity of $\mathcal{E}^1_C(M,\Omega)$ directly follows from the concavity of $E$. The compactness is a consequence of Proposition \ref{prop_Echi_incr} and Lemma \ref{lem:rel_cptness} as $\mathcal{E}^1_C(M,\Omega)$ is closed and contained in the set 
    \[
    \left\{ \phi \in \QPSH(M,\Omega) \mid -C'\leq \sup_M \phi \leq 0 \right\}
    \]
    for a certain $C'>0$ depending on $C>0$.
\end{proof}


\section{Existence of weak solutions}\label{sec 5}
The purpose of this section is to prove the following

\begin{thm}\label{Thm:existence}
Let $(M,I,J,K,\Omega)$ be a compact locally flat HKT manifold such that $K_I(M)$ is holomorphically trivial. Then equation \eqref{eq_QMA} admits a solution $\varphi\in \mathcal{E}^1(M, \Omega)$.
\end{thm}

\smallskip

Following \cite{Berman-Boucksom-Guedj-Zeriahi}, we use a variational approach to prove Theorem \ref{Thm:existence}. We define the \emph{Ding functional} $  \mathcal{F}\colon \mathcal{E}^1(M,\Omega) \to \R $ as
$$
\mathcal{F}(\phi):=E(\phi)-\int_M \phi\, {\rm e}^F\Omega^n\,,
$$
It is immediate to check that $\mathcal{F}(\varphi+c)=\mathcal{F}(\varphi)$ for any $\varphi\in\mathcal{E}^1(M,\Omega)$ and $c\in \R$, and that $\mathcal{F}$ is upper semicontinuous as a consequence of Proposition \ref{prop_Echi_incr}.
Note also that formula \eqref{eq_primitive} implies that for a smooth path $ \phi_t \colon [0,1] \to \mathcal{E}^1(M,\Omega)\cap C^\infty(M) $ we have 
\[
\frac{d}{dt} \mathcal{F}(\phi_t)=
\int_M \dot \phi_t\, \MA_{\phi_t}  - 
\int_M \dot \phi_t\, {\rm e}^F\Omega^n \,.
\]
In particular $\phi$ is a critical point for $ \mathcal{F} $ over $\mathcal{E}^1(M,\Omega)\cap C^\infty(M)$ if and only if 
\begin{equation}\label{eqn:QMA}
\MA_\phi= {\rm e}^F\Omega^n\,.
\end{equation}
The strategy to prove Theorem \ref{Thm:existence} consists in showing that maximisers for $\mathcal{F}$ in $\mathcal{E}^1(M,\Omega)$ are still solutions to \eqref{eqn:QMA}.
\smallskip

Here we introduce the {\em $ \Omega $-quaternionic plurisubharmonic envelope} $ P(\psi)\in\QPSH(M,\Omega) $  of an upper semicontinuous function $\psi$ as 
\[
P(\psi)(x):= \sup \{\phi(x)\mid \phi \in \QPSH(M,\Omega)\,, \phi \leq \psi \}\,.
\]
Observe that $ P(\psi) $ is upper semicontinuous: indeed $ P(\psi)\leq \psi $ implies $ P(\psi)^*\leq \psi^*=\psi $, i.e. $ P(\psi)^* $ is a competitor in the definition of $ P(\psi) $, thus $ P(\psi)=P(\psi)^* $.


\begin{lem}\label{Lem:support_MA_P(psi)}
	If $\psi$ is continuous,  $\MA_{P(\psi)} $ is supported on $ \{P(\psi)=\psi\} $.
\end{lem}
\begin{proof}
	Since $ \psi $ is continuous and $ P(\psi) $ is upper semicontinuous the set $ \{ P(\psi)<\psi \} $ is open. The lemma is then a consequence of the fact that we can solve the Dirichlet problem in sufficiently small balls (see \cite[Lem. 3.2]{Wan-Kang}) and a standard balayage procedure performed on small balls inside $ \{ P(\psi)<\psi \} $ (cf. \cite[Prop. 4.1, Thm. 5.2.(2)]{Guedj-Zeriahi 2005}).
\end{proof}

One key result to prove is the following Projection Theorem. The idea in the complex setting is due to Berman and Boucksom \cite{Berman-Boucksom}, the proof was later simplified by Lu and Nguyen \cite{Lu-Nguyen}, which we adapt closely to our case.

\begin{prop}\label{Prop:projection_thm}
	For every $ \phi \in \mathcal{E}^1(M,\Omega) $ and $ v\in C^0(M) $
	\begin{equation}\label{eq:Proj}
	\frac{d}{dt} E (P(\phi+tv))\big\vert_{t=0}=\int_M v\, \MA_\phi\,.
	\end{equation}
\end{prop}
\begin{proof}
First, we observe that it is enough to prove that
	\begin{equation}
    \label{eq:ProjNew}
	E(P(\psi+u))-E(P(\psi))=\int_0^1\left(\int_M u\, \MA_{P(\psi+tu)}\right) dt
	\end{equation}
	for every pair of continuous functions $ \psi,u $ on $ M $. Indeed, take a sequence $ (\phi_j)_{j\in \N} $ of continuous functions on $ M $ that decrease to $ \phi $.  Such a sequence exists because $ \phi $ is upper semicontinuous (observe that the $ \phi_j $'s need not be in $ \QPSH(M,\Omega) $).
	By continuity along decreasing sequences of the Monge--Ampère energy (Proposition \ref{prop_Echi_incr}) and the Monge--Amp\`{e}re operator (see Lemma \ref{propE} below), we get
	\[
	E(P(\phi+u))-E(\phi)=\lim_{j\to +\infty}\bigl(E(P(\phi_j+u))-E(P(\phi_j))\bigr)
	\]
    as $P(\phi)=\phi, P(\phi_j)\searrow P(\phi), P(\phi_j+u)\searrow P(\phi+u)$, and
	\[
	\int_0^1\left(\int_M u \,\MA_{P(\phi+tu)}\right)dt=\lim_{j\to +\infty}\int_0^1\left(\int_M u\, \MA_{P(\phi_j+tu)}\right)dt\,.
	\]
    Therefore from \eqref{eq:ProjNew} we obtain
    \begin{equation}\label{eq:DagliUnNome}
    E(P(\varphi+u))-E(\varphi)=\int_0^1\left(\int_M u\,\MA_{P(\phi+tu)}\right)dt
    \end{equation}
    for any continuous function $u$. In particular, assuming first $v\geq 0$, \eqref{eq:Proj} follows from \eqref{eq:DagliUnNome} and Lemma \ref{propE} by replacing $u$ with $sv$, dividing by $s$, and taking the limit as $s\to 0$. For general $v\in C^0(M)$ we can replace $v$ by $v-\inf_M v$, observing that
    $$
    \frac{d}{dt}E(P(\varphi+t(v+C)))_{|t=0}= \frac{d}{dt}E(P(\varphi+tv))_{|t=0}+C
    $$
    for any $C\in \R$ (since $P(f+C)=P(f)+C$ by definition) and that $\MA_\varphi(M)=1$.
	
	It remains to prove \eqref{eq:ProjNew} for $ \psi \in C^0(M) $, which in turn we claim will follow from
    \begin{equation}\label{eqn:New}
        \frac{d}{dt} E( P(\psi+tu))_{|t=0^+}=\int_M u \MA_{P(\psi)}.
    \end{equation}
    Indeed, replacing $\psi$ by $\psi+su$ and the direction $u$ by $-u$, from \eqref{eqn:New} we obtain
    $$
    \frac{d}{dt} E(P(\psi+tu))_{|t=s}=\int_M u \MA_{P(\psi+su)}
    $$
    for any $s\in [0,1]$. Integrating over the interval $[0,1]$ we deduce $\eqref{eq:ProjNew}$.
    Let us now prove \eqref{eqn:New}. From the concavity of the energy we deduce
	\begin{align*}
	E( P(\psi+tu))&\leq E(P(\psi))+E'(P(\psi))(P(\psi+tu)-P(\psi))\,,\\
	E (P(\psi))&\leq E(P(\psi+tu))+E'(P(\psi+tu))(P(\psi)-P(\psi+tu))\,,
	\end{align*}
	which, together with \eqref{eq_primitive}, give
	\[
		\int_M \frac{P(\psi+tu)-P(\psi)}{t}\MA_{P(\psi+tu)}\leq \frac{E( P(\psi+tu))-E(P(\psi))}{t}\leq \int_M \frac{P(\psi+tu)-P(\psi)}{t}\MA_{P(\psi)}\,.
	\]
	Using Lemma \ref{Lem:support_MA_P(psi)} and the inequality $ P(\psi+tu)\leq \psi+tu $ we then obtain
	\begin{equation}\label{eq_proj_thm1}
		\int_M u\,\MA_{P(\psi+tu)}\leq \frac{E (P(\psi+tu))-E(P(\psi))}{t}\leq \int_M u\,\MA_{P(\psi)}\,.
	\end{equation}
	Observe that the projection is uniformly Lipschitz as, by definition $P(\psi+u)\leq P(\psi+\sup_M \lvert u\rvert)=P(\psi)+\sup_M \lvert u\rvert$, and similarly $P(\psi +u)\geq P(\psi)-\sup_M \lvert u\rvert$ for any $u\in C^0(M)$. Thus we see that
	\[
	\sup_M |P(\psi+tu)-P(\psi)|\leq t\sup_M |u|\,.
	\]
	It follows that $ P(\psi+tu) \to P(\psi) $ uniformly as $ t\to 0^+ $.
	By continuity of the Monge--Amp\`{e}re operator of Lemma \ref{eqn:Conv_unif} below, taking the limit in \eqref{eq_proj_thm1} as $ t\to 0^+ $ yields \eqref{eqn:New}, as required.
\end{proof}

\begin{lem}\label{eqn:Conv_unif}
    Let $\phi\in \QPSH(M,\Omega)\cap C^0(M)$ and let $(\phi_j)_{j\in \N}\subset \QPSH(M,\Omega)\cap C^0(M)$ be such that $\phi_j\to \phi$ uniformly. Then $\MA_{\phi_j}\to \MA_{\phi}$ weakly, as $j\to +\infty$.
\end{lem}
\begin{proof}
Approximating uniformly $f\in C^0(M)$ by $f_k\in C^\infty(M)$ we observe that
\begin{align*}
    \left\lvert \int_M f (\MA_{\varphi_j} - \MA_\varphi)\right\rvert&\leq \left\lvert \int_M (f-f_k)\MA_{\varphi_j} \right\rvert +\left\lvert \int_M f_k (\MA_{\varphi_j} - \MA_\varphi)\right\rvert + \left\lvert \int_M (f-f_k)\MA_\varphi \right\rvert \\
    &\leq 2\left\lVert f-f_k \right\rVert_\infty + \left\lvert \int_M f_k (\MA_{\varphi_j} - \MA_\varphi)\right\rvert
\end{align*}
as $\MA_{\varphi_j}(M)=\MA_{\varphi}(M)=1$. Thus it suffices to show that 
$$
\int_M f \MA_{\varphi_j}\to \int_M f\MA_\varphi\,
$$
for every $f\in C^{\infty}(M)$. 

    By using a partition of unity we can reduce to the case when $f$ has compact support in the domain $A$ of a coordinate chart. Then, letting $u$ be a local potential for $\Omega$ in $A$, we have $\MA_{\phi_j}=(\partial \partial_J \psi_j)^n$, and $\MA_\phi=(\partial \partial_J \psi)^n$, where we set $\psi_j:=u+\phi_j, \psi:= u+\phi \in \QPSH(A)$. Clearly $\lVert\psi_j-\psi \rVert_{\infty, K}\leq \lVert \phi_j-\phi \rVert_\infty$ for any compact set $K\subset A$, for instance $K=\mathrm{Supp}(f)$.
    Let $k=1,\dots,n$, and fix a subdomain $A'\Subset A$ such that $\mathrm{Supp}(f)\Subset A'$. Then, we have
    $$
    \begin{aligned}
        &\left\lvert\int_A f(\partial \partial_J \psi_j)^k\wedge (\partial \partial_J\psi)^{n-k}-\int_A f(\partial \partial_J \psi_j)^{k-1}\wedge (\partial \partial_J \psi)^{n-k+1}  \right\rvert\\
        &\qquad = \left\lvert \int_A (\psi_j-\psi)(\partial\partial_J \psi_j)^{k-1}\wedge (\partial \partial_J \psi)^{n-k} \wedge \partial \partial_J f \right\rvert\\
        &\qquad \leq C \lVert\phi_j-\phi \rVert_{\infty} \int_{A'} (\partial \partial_J\psi_j)^{k-1}\wedge (\partial \partial_J \psi)^{n-k} \wedge \partial\partial_J u\\
        &\qquad \leq C \lVert\phi_j-\phi \rVert_{\infty} \int_M(\Omega+\partial\partial_J \phi_j)^{k-1}\wedge (\Omega+\partial \partial_J \phi)^{n-k}\wedge \Omega =C\lVert \phi_j-\phi\rVert_{\infty}\,,
    \end{aligned}
$$
where $-C\partial\partial_J u\leq \partial\partial_J f\leq C \partial \partial_J u$ and in the last equality we used \eqref{eq:MAmass}.
We deduce that
$$
\left\lvert \int_A f (\partial \partial_J\psi_j)^n-\int_A f(\partial \partial_J\psi)^n \right\rvert\leq n C \lVert \phi_j-\phi\rVert_\infty\,,
$$
which concludes the proof.
\end{proof}

\begin{lem}\label{propE}
Let $(\phi_j)_{j\in \N} $ be a decreasing sequence in $ \mathcal{E}^1(M,\Omega)$ converging to $\phi \in \mathcal{E}^1(M,\Omega)$. Then $ \MA_{\phi_j}\to \MA_\phi $ weakly, as $ j\to+ \infty $.
\end{lem}
\begin{proof}
Assume without loss of generality that $\phi\leq 0$ and $\phi_j\leq 0$ for all $j$. Consider the canonical approximants $\phi^k:=\max(\phi,-k)$, $\phi_j^k:=\max(\phi_j,-k)$. For any fixed $k$ we clearly have $\phi_j^k \searrow \phi^k$ and from continuity of the Monge--Ampère operator along decreasing sequences in $\QPSH(M,\Omega)\cap L^\infty(M)$ (Theorem \ref{thm:cont}) we deduce that $\MA_{\phi_j^k} \to \MA_{\phi^k}$.

Let $C>0$ be such that $E(\varphi)\geq -C$. We claim that
\begin{equation}\label{eqn:contE1}
\int_{\{\varphi_j\leq -k\}}\MA_{\varphi_j}\leq \frac{n+1}{k}C
\end{equation}
for any $j,k\in\N$. Indeed, by Proposition \ref{prop_Echi_incr} we have
$$
-C\leq E(\varphi)\leq E(\varphi_j)\leq \frac{1}{n+1}\int_M \varphi_j \MA_{\varphi_j}\leq \frac{1}{n+1}\int_{\{\varphi_j\leq -k\}}\varphi_j\MA_{\varphi_j}\leq \frac{-k}{n+1}\int_{\{\varphi_j\leq -k\}}\MA_{\varphi_j}
$$
and the claim \eqref{eqn:contE1} follows. Then we have
\begin{multline*}
    \left\lvert\int_M \chi \left(\MA_{\varphi_j}-\MA_{\varphi}\right) \right\rvert\\
    \leq \underbrace{\left\lvert  \int_M\chi \left(\MA_{\varphi_j}-\MA_{\varphi_j^k}\right) \right\rvert}_{\mathrm{(I)}_{j,k}} + \left \lvert  \int_M\chi \left(\MA_{\varphi_j^k}-\MA_{\varphi^k}\right) \right \rvert +\underbrace{\left \lvert  \int_M\chi \left(\MA_{\varphi^k}-\MA_{\varphi}\right) \right \rvert}_{\mathrm{(II)}_{k}}\,.
\end{multline*}
The second term on the right-hand side tends to $0$ as $j\to +\infty$ thanks to the first part of the proof. Thus, we need to show that $\limsup_k\limsup_j\mathrm{(I)}_{j,k}+\mathrm{(II)}_{k}=0$.

For any $j,k\in\N$ we have
\[
\int_{\{\varphi_j\leq -k\}}\MA_{\varphi_j}=\int_M \MA_{\varphi_j}-\int_{\{\varphi_j>-k\}}\MA_{\varphi_j}= \int_M \MA_{\varphi_j^k}-\int_{\{\varphi_j>-k\}}\MA_{\varphi_j^k}=\int_{\{\varphi_j\leq -k\}}\MA_{\varphi_j^k}\,,
\]
where the second equality follows since $\varphi_j, \varphi_j^k\in\mathcal{E}(M,\Omega)$ and the maximum principle of Lemma \ref{lem:max_princ_Unbounded}. Combining this with \eqref{eqn:contE1} we obtain
\begin{align*}
    \mathrm{(I)}_{j,k}\leq \lVert \chi\rVert_{\infty}\int_{\{\varphi_j\leq -k\}}\left(\MA_{\varphi_j}+\MA_{\varphi_j^k}\right)=2\lVert\chi \rVert_{\infty}\int_{\{\varphi_j\leq -k\}}\MA_{\varphi_j}\leq \frac{2(n+1)C\lVert \chi\rVert_\infty }{k}\,,
\end{align*}
and similarly
$$
\mathrm{(II)}_k\leq 2\lVert \chi \rVert_{\infty}\int_{\{\varphi\leq -k\}}\MA_{\varphi}\,.
$$
We deduce that $\limsup_k\limsup_j\mathrm{(I)}_{j,k}+\mathrm{(II)}_k\leq 0 $ as $\MA_\varphi$ does not put mass on the Q-polar set $\{\varphi=-\infty\}$ (Proposition \ref{prop:NonPluripolar}).
\end{proof}

We can now observe that maximisers of $\mathcal{F}$ solve \eqref{eqn:QMA}.
\begin{prop}\label{Prop:final}
    Let $\varphi\in\mathcal{E}^1(M,\Omega)$ be such that $\mathcal{F}(\varphi)=\sup_{\mathcal{E}^1(M,\Omega)}\mathcal{F}$. Then $\varphi$ solves \eqref{eqn:QMA}.
\end{prop}
\begin{proof}
Let $v\in C^0(M)$ and consider $t\to P(\varphi+tv)$. By Proposition \ref{Prop:projection_thm} it follows that 
\begin{equation}
    \label{eqn:Key}
    \frac{d}{dt}\left( E(P(\varphi+tv))-\int_M (\varphi+tv)\mathrm{e}^F \Omega^n\right )\big\vert_{t=0}=\int_M v (\MA_\varphi- \mathrm{e}^F \Omega^n)\,.
\end{equation}
Thus the trivial inequality $P(\varphi+tv)\leq \varphi+tv$ gives
$$
E(P(\varphi+tv))-\int_M (\varphi+tv)\mathrm{e}^F \Omega^n\leq \mathcal{F}(\varphi)\,,
$$
i.e. $t\to E(P(\varphi+tv))-\int_M(\varphi+tv)\mathrm{e}^F \Omega^n$ has a maximum for $t=0$. Hence $\int_M v(\MA_\varphi-\mathrm{e}^F \Omega^n)=0$ for any $v\in C^0(M)$, i.e. $\MA_\varphi=\mathrm{e}^F \Omega^n$.
\end{proof}

Finally, the following properness of the Ding functional will ensure the existence of a maximiser. 

\begin{lem}\label{lem:properness}
$ \mathcal{F} $ is proper with respect to $ E $. More precisely there is a constant $ C>0 $ such that
\begin{equation}
    \label{eqn:Properness}
    \mathcal{F}(\phi)\leq E(\phi)-\sup_M \phi +C \left \lvert E(\phi)-\sup_M \phi \right \rvert^{1/2}\,,
\end{equation}
for all $ \phi \in \mathcal{E}^1(M,\Omega) $.
\end{lem}
\begin{proof}

Let $ \phi \in \mathcal{E}^1(M,\Omega) $. 
Without loss of generality assume $ E(\phi)\leq -1 $. Moreover by continuity of the Ding functional and of the Monge--Ampère energy along decreasing sequences (Proposition \ref{prop_Echi_incr}) and by translation invariance of the quantities in \eqref{eqn:Properness}, we can and will suppose that $\sup_M \phi=0$ and that $\phi\in \mathcal{E}^1(M,\Omega)\cap L^\infty(M)$.

Set $ \epsilon=|E(\phi)|^{-1/2} $ so that $ \psi=\epsilon \phi $ is still $ \Omega $-quaternionic psh. For any $ 1\leq j \leq n  $ we have
\[
\Omega_\psi^j\wedge \Omega^{n-j}= \Omega^n+\sum_{k=1}^{j}\binom{j}{k}\epsilon^k(\partial \partial_J \phi)^k \wedge \Omega^{n-k}\leq \Omega^n+N\epsilon \sum_{k=0}^{n}\Omega_\phi^k \wedge \Omega^{n-k}
\]
for some $ N\in \N $, therefore, since $\{\varphi\in \QPSH(M,\Omega) \mid \sup_M \varphi=0\}$ is compact by Lemma \ref{lem:rel_cptness}, we have
\[
E(\psi)=\frac{\epsilon}{n+1
}\sum_{j=0}^n\int_M \phi\, \Omega_\psi^j \wedge \Omega^{n-j} \geq 
\frac{\epsilon}{n+1}\int_M \phi\, \Omega^n + N\epsilon^2 E(\phi)\geq \frac{1}{n+1}\int_M\phi \, \Omega^n-N\geq -\tilde C\,,
\]
proving that $ \psi \in \mathcal{E}_{\tilde C}^1(M,\Omega) $. Moreover the same Lemma \ref{lem:rel_cptness} gives $ \int_M \psi\, \mathrm{e}^F \Omega^n\geq -C' $ for every $ \psi \in \mathcal{E}^1_{\tilde C}(M,\Omega) $. Therefore
\[
\int_M \phi \,\mathrm{e}^F \Omega^n =|E(\phi)|^{1/2}\int_M \psi \, \mathrm{e}^F \Omega^n\geq -C'|E(\phi)|^{1/2}\,,
\]
which clearly gives \eqref{eqn:Properness} and concludes the proof.
\end{proof}

We can now prove the existence of weak solutions to \eqref{eq_QMA}.

\begin{proof}[Proof of Theorem $\ref{Thm:existence}$]
    By translation invariance and properness of $ \mathcal{F} $ (Lemma \ref{lem:properness}) there exists $ C>0 $ large enough to ensure
	\[
	\sup_{\mathcal{E}^1(M,\Omega)}\mathcal{F}=\sup_{\mathcal{E}^1_C(M,\Omega)}\mathcal{F}\,.
	\]
    Let then $(\varphi_j)_{j\in \N} \subset\mathcal{E}^1_C(M,\Omega)$ be such that $\lim_{j\to + \infty}\mathcal{F}(\varphi_j)=\sup_{\mathcal{E}^1(M,\Omega)}\mathcal{F}$. By compactness of $\mathcal{E}^1_C(M,\Omega)$ (Proposition \ref{prop:cptness}) we can and will assume that $\varphi_j\to \varphi\in \mathcal{E}^1_C(M,\Omega)$. Thus the upper semicontinuity of $\mathcal{F}$ (Proposition \ref{prop_Echi_incr}) gives
    $$
    \mathcal{F}(\varphi)\geq \limsup_{j\to +\infty } \mathcal{F}(\varphi_j)=\sup_{\mathcal{E}^1(M,\Omega)}\mathcal{F}\,,
    $$
    i.e. $\varphi$ is a maximiser of $\mathcal{F}$. Proposition \ref{Prop:final} concludes the proof.
\end{proof}

\section{Quaternionic capacity and boundedness of solutions}\label{Sec 6}
The purpose of this section is twofold: introducing a  Monge--Amp\`{e}re capacity in the quaternionic setting and proving the following 

\begin{thm}\label{Thm:boundedness}
Let $(M^n,I,J,K,\Omega)$ be a compact locally flat HKT manifold such that $K_I(M)$ is holomorphically trivial. Then any solution $\phi\in  \mathcal{E}^1(M,\Omega)$
to \eqref{eq_QMA} is bounded. 
\end{thm}


For any Borel subset $ E\subseteq M $ we define the \emph{quaternionic Monge--Amp\`{e}re capacity}
\begin{equation}\label{eq_capacity}
	\mathrm{Cap}_\Omega(E):=\sup \left\{ \int_E \MA_u \mid u \in \QPSH(M,\Omega)\,, \,\,0\leq u \leq 1 \right\}\,.
\end{equation}
We extend the definition of  $\mathrm{Cap}_\Omega$ to arbitrary subsets $ E\subseteq M $ by
\[
	\mathrm{Cap}_\Omega(E):=\sup \{ \mathrm{Cap}_\Omega(K) \mid K\subseteq E \text{ is compact} \}\,.
\]
It is not difficult to show that the capacity satisfies the following properties:\begin{enumerate}
	\item {\em Monotonicity}: if $ E_1\subseteq E_2\subseteq M $ are Borel subsets, then
		\[
		\mathrm{Vol}(E_1)\leq \mathrm{Cap}_\Omega(E_1)\leq \mathrm{Cap}_\Omega(E_2)\leq \mathrm{Cap}_\Omega(M)=\mathrm{Vol}(M)=1\,.
		\]

\item {\em Subadditivity}: if $\{E_j\}_{j\in\N}$ is a family of Borel subsets of $ M $, then $ \mathrm{Cap}_\Omega(\bigcup E_j)\leq \sum \mathrm{Cap}_\Omega(E_j) $. Moreover, if $E_j\subseteq E_{j+1}$, then $ \mathrm{Cap}_\Omega(\bigcup E_j)=\lim_{j\to +\infty} \mathrm{Cap}_\Omega(E_j)$.
\end{enumerate}

In \cite{Wan-Kang}  Wan and Kang introduced the following \emph{relative Monge--Amp\`{e}re capacity} on a domain $A$ of $\H^n$:
\[
\mathrm{Cap^{WK}}(E,A):=\sup \left\{ \int_E (\partial \partial_J u)^n\mid u\in  \QPSH(A)\,, \,\, 0\leq u  \leq 1 \right\}\,,
\]
where $ E\subseteq A $ is a Borel set. Since we are assuming $(M,I,J,K)$ locally flat,  $\mathrm{Cap^{WK}}(E,A)$ induces a capacity on $M$ as follows. Let $ U_1,\dots, U_N$ be a finite open cover of $M$ made by subsets of the form $U_j=\{x\in M \mid \rho_j(x)<0 \}$, where $ \rho_j $ is a smooth strictly quaternionic psh function on a neighbourhood of $ \bar U_j $. Choose another open cover $ V_1,\dots,V_N $ such that $ V_j\subseteq U_j $
and define
\[
\mathrm{Cap^{WK}}(E):=\sum_{j=1}^N \mathrm{Cap^{WK}}(E\cap V_j,U_j)\,,
\]
for every Borel subset $ E\subseteq M$. The next lemma says that the quaternionic Monge--Ampère capacity and the Wan-Kang capacities are comparable. This is analogous to the K\"ahler case (see \cite[Inequality (1.1)]{Kolodziej} and also \cite[Prop. 2.10]{Guedj-Zeriahi 2005}). 

\begin{lem}\label{lem:Cap_C_comparable}
	There exists a constant $ C \geq 1 $ such that for every Borel subset $ E\subseteq M $
	\[
	 C^{-1}\mathrm{Cap^{WK}}(E)\leq \mathrm{Cap}_\Omega(E)\leq C \cdot \mathrm{Cap^{WK}}(E)\,.
	\]
\end{lem}
\begin{proof}
From \cite[Prop. 3.3]{Wan-Kang} we have
	\[
	\mathrm{Cap^{WK}}(E\cap V_j,U_j)=\int_{E\cap V_j}(\partial \partial_J v^*_j)^n\,,
	\]
	where we set
	\[
	v_j:=
	\sup \{ u \text{ quaternionic psh in }U_j \mid  u\leq 0, u\vert_{E\cap V_j}\leq -1 \}\,.
	\]
Observe that the upper semicontinuous regularisation $ v_j^* $ is quaternionic psh by Proposition \ref{prop_basic} and satisfies $ -1\leq v_j^* \leq 0 $ as well as $ v_j^*=0 $ on $ \partial U_j $. The lemma follows by the same argument as in \cite[p. 670]{Kolodziej}.
\end{proof}

In view of the last Lemma, we can prove the following control of the volume with the quaternionic Monge--Ampère capacity.

\begin{lem}\label{Lem:volcap}
For every $\alpha \in (0,1)$ there exists a constant $C$ such that for every Borel subset $E \subseteq M$ we have
\[
\int_E \Omega^n \leq C \cdot \mathrm{Cap}_\Omega(E)^{1+\alpha}\,.
\]
\end{lem}
\begin{proof}
Take a finite cover of $M$ of open balls $\{\mathbb{B}_j\}_{j=1}^N$ of radius $R>0$. Since $M$ is locally flat we may further assume that $\Omega^n\vert_{\mathbb{B}_j} \leq C_1 d\lambda $, for a  uniform constant $C_1>0$, where $\lambda$ denotes the Lebesgue measure. Thanks to \cite[Lem. 9]{Sro}, for any $\alpha \in (0,1)$ there exists a constant $C_2$, depending only on $\alpha$ and the radius $R$, such that $\lambda(E\cap \mathbb{B}_j) \leq C_2 \mathrm{Cap^{WK}}(E\cap \mathbb{B}_j)^{1+\alpha}$. Then, thanks to Lemma \ref{lem:Cap_C_comparable} there exists a constant $C_3\geq 1$ such that
\[
\begin{split}
\int_{E}\Omega^n &\leq \sum_{j=1}^N \int_{E\cap \mathbb{B}_j} \Omega^n \leq C_1C_2 \sum_{j=1}^N \mathrm{Cap^{WK}} (E \cap \mathbb{B}_j)^{1+\alpha}\\
&\leq C_1C_2C_3 \sum_{j=1}^N{\rm Cap}_\Omega(E\cap \mathbb{B}_j)^{1+\alpha}\leq C_1C_2C_3 N {\rm Cap}_\Omega(E)^{1+\alpha}\,,
\end{split}
\]
where the last inequality is due to the monotonicity of the capacity.
\end{proof}

Next, we will need the following extension of the comparison principle to $\mathcal{E}(M,\Omega)$.

\begin{prop}\label{Comparison}
    For any $\phi,\psi \in\mathcal{E}(M,\Omega)$
	\[
	\int_{\{\phi>\psi\}}\MA_\phi\leq \int_{\{\phi>\psi\}}\MA_\psi\,.
	\]
\end{prop} 
\begin{proof}
Consider the canonical approximants $ \phi_k:=\max(\phi,-k),\psi_j:=\max(\psi,-j) $. By the comparison principle for bounded $ \Omega $-quaternionic psh functions (Proposition \ref{Prop:cps_principle}) we have
\begin{equation}\label{eqn:incomp}
\int_{\{\phi>\psi_j\}} \MA_{\phi_k} \leq \int_{\{\phi_k>\psi_j\}} \MA_{\phi_k} \leq \int_{\{\phi_k>\psi_j\}} \MA_{\psi_j} \leq \int_{\{\phi_k>\psi\}} \MA_{\psi_j}\,.
\end{equation}
Note that for any Borel set $B$,
\[
\begin{aligned}
\lim_{k\to +\infty}\int_{B }\MA_{\phi_k}&=\lim_{k\to +\infty }\left( \int_{B \cap \{ \phi> -k\}}\MA_{\phi_k}+\int_{B\cap \{ \phi\leq -k\}}\MA_{\phi_k}\right)=\int_{B}\MA_{\phi}\,, 
\end{aligned}
\]
since $\phi \in \mathcal{E}(M,\Omega)$, and by definition
\[
\lim_{k\to +\infty} \int_{\{\phi\leq -k\}} \MA_{\phi_k}=1-\lim_{k\to +\infty} \int_{\{\phi>-k\}} \MA_{\phi_k} =1-\int_M \MA_\phi =0\,. 
\]
Similarly, $\lim_{j\to +\infty}\int_{B} \MA_{\psi_j}=\int_{B} \MA_{\psi}$. Thus, letting first $j\to +\infty$ and then $k\to +\infty$ in \eqref{eqn:incomp} yields
\begin{align*}
    \int_{\{\varphi>\psi\}}\MA_\varphi&=\lim_{k\to +\infty}\int_{\{\varphi>\psi\}}\MA_{\phi_k}=\lim_{k\to +\infty}\lim_{j\to +\infty}\int_{\{\varphi>\psi_j\}}\MA_{\phi_k}\\
    &\leq \lim_{k\to +\infty}\lim_{j\to +\infty} \int_{\{\varphi_k>\psi\}}\MA_{\psi_j}=\lim_{k\to +\infty}\int_{\{\varphi_k>\psi\}}\MA_\psi\leq \int_{\{\varphi\geq \psi\}}\MA_\psi.
\end{align*}
Replacing $ \psi $ with $ \psi+\epsilon $ and letting $ \epsilon $ decrease to $ 0 $ we conclude.
\end{proof}

\begin{lem}\label{Lem_CapMA}
Let $\phi \in \mathcal{E}^1(M,\Omega)$, then for any $s>0$ and $0\leq t \leq 1 $, we have
\[
t^n \mathrm{Cap}_\Omega ( \phi<-s-t) \leq \int_{\{\phi<-s\}} \MA_\phi\,.
\]
\end{lem}
\begin{proof}
Let $u\in \QPSH(M,\Omega)$ be such that $0\leq u \leq 1$. Set $v:=tu-s-t\in \mathcal{E}^1(M,\Omega)$. We have
\[
\begin{split}
t^n\int_{\{\phi<-s-t\}}(\Omega + \partial \partial_J u)^n&\leq \int_{\{\phi<-s-t\}}\left(\Omega+\partial \partial_J (tu)\right)^n=
\int_{\{\phi<-s-t\}}(\Omega+ \partial \partial_J v)^n\,.
\end{split}
\]
Observe that we have the inclusions
\[
\{\phi<-s-t\} \subseteq \left\{ \phi < v \right\} \subseteq \{\phi <-s\}
\]
and thus
\[
t^n\int_{\{\phi<-s-t\}}(\Omega + \partial \partial_J u)^n\leq \int_{\{\phi<v\}}(\Omega+ \partial \partial_J v)^n\,.
\]
Furthermore, applying the comparison principle of Proposition \ref{Comparison} we deduce
\[
\int_{\{\phi<v\}}(\Omega+ \partial \partial_J v)^n\leq \int_{\{\phi<v\}}(\Omega+ \partial \partial_J \phi)^n\leq \int_{\{\phi<-s\}}(\Omega+ \partial \partial_J \phi)^n\,,
\]
as claimed.
\end{proof}

Before we prove boundedness of solutions to the quaternionic Monge--Ampère equation we recall the statement of a general lemma due to de Giorgi, later generalised by Ko\l odziej \cite{K98} (see also {\cite[Lem. 2.4]{EGZ}}).

\begin{lem}[de Giorgi]\label{Lem:EGZ}
Let $f \colon \R^+ \to \R^+$ be a decreasing right-continuous function such that $\lim_{s\to +\infty}f(s)=0$. Suppose also that there exist $\alpha>0$ and $B>0$ such that
\begin{equation}\label{eq_boundedness}
tf(s+t)\leq B f(s)^{1+\alpha}\,, \qquad \text{for all }s>0\,, \,\, 0\leq t\leq 1\,.
\end{equation}
Then there exists $S>0$, depending only on $\alpha, B$ and the minimum $s_0$ such that $f(s_0)^\alpha<(2B)^{-1}$, which satisfies $f(s)=0$ for all $s\geq S$. Furthermore we have
\begin{equation}\label{eqn:S}
S\leq s_0+\frac{2B}{1-2^{-\alpha}}f(s_0)^\alpha\, .
\end{equation}
\end{lem}

We are ready to prove the aforementioned boundedness result.

\begin{proof}[Proof of Theorem $\ref{Thm:boundedness}$]
Suppose without loss of generality that $\sup \phi =0$. Consider the function
\[
f(s):= {\rm Cap}_\Omega(\phi<-s)^{1/n}\,.
\]
We will observe that $f$ satisfies all the assumptions of Lemma \ref{Lem:EGZ}. Clearly $f$ is right-continuous and decreasing. Moreover,
$\lim_{s\to +\infty}f(s)=0$ as by Lemma \ref{Lem_CapMA}
$$
\limsup_{s\to +\infty} f(s)^n\leq \lim_{s\to +\infty}\int_{\{\varphi<-s+1\}}\MA_\varphi=\int_{\{\varphi=-\infty\}}\MA_\varphi=0
$$
because the Monge--Ampère operator puts no mass on Q-polar sets by Proposition \ref{prop:NonPluripolar}. It then remains to prove \eqref{eq_boundedness}. 
By applying again Lemma \ref{Lem_CapMA} and using that $\phi$ solves the quaternionic Monge--Ampère equation we have
\[
t^nf(s+t)^n\leq \int_{\{\phi<-s\}}(\Omega+ \partial \partial_J \phi)^n=\int_{\{\phi<-s\}}\mathrm{e}^F\Omega^n\leq (\sup \mathrm{e}^F) \int_{\{\phi<-s\}}\Omega^n\,.
\]
Finally, applying Lemma \ref{Lem:volcap}, for any $\alpha \in (0,1)$ we have
\[
\int_{\{\phi<-s\}} \Omega^n \leq C \cdot {\rm Cap}_\Omega(\phi<-s)^{1+\alpha}=Cf(s)^{n(1+\alpha)}\,.
\]
Therefore \eqref{eq_boundedness} is satisfied.

We can now apply Lemma \ref{Lem:EGZ} and conclude that there exists $S=S(\alpha,B)>0$, such that $f(S)=0$. As a consequence the set $\{\phi <-S\}$ has zero measure, thanks to Lemma \ref{Lem:volcap}. By upper semicontinuity of $\phi$, this directly implies that $\phi$ must be bounded. 
\end{proof}

\section{Continuity of solutions}\label{Sec 7}

The aim of this section is to prove the following 

\begin{thm}\label{Thm:continuity} 
Let $(M^n,I,J,K,\Omega)$ be a compact locally flat HKT manifold such that $K_I(M)$ is holomorphically trivial. Then any  solution $\varphi \in  \mathcal{E}^1(M,\Omega)$ to \eqref{eq_QMA} is continuous.  
\end{thm}

To prove Theorem \ref{Thm:continuity} we strictly follow the argument in \cite[Sec. 2.4]{K98} (see also \cite[Thm. 1.3]{Cho-Choi}).

\medskip 
We need the following preliminary

\begin{lem}\label{Lem:potential}
Let $(M,I,J,K,\Omega)$ be a locally flat HKT manifold and $x_0 \in M$. There exist open balls $\mathbb{B}'\subseteq \mathbb{B}$ about $0$ in $\mathbb H^n$, a local chart $\eta$ defined on a neighbourhood $U$ of $x_0$ onto $\mathbb{B}$ such that $\eta(x_0)=0$, and a function $u\in C^\infty(\mathbb{B})$ satisfying $\eta^* \partial \partial_J u=\Omega\vert_U$ and $u(x)-u(x_0)>0$ for all $x\in \mathbb{B}'\setminus \{x_0\}$. 
\end{lem}
\begin{proof}
    Let $\eta=(z_1,\dots,z_{2n})$ be a coordinate chart defined on a neighbourhood $U$ of $x_0$ onto $\mathbb{B}$ such that $\eta(x_0)=0$ and let $v\in C^\infty(\mathbb{B})$ be a strictly quaternionic psh function such that $\eta^* \partial \partial_J v= \Omega\vert_U$. 
    
    Let $h\colon \mathbb{B} \to \R$ be the real part of the function
    \[
    f(z):=-2 \sum_{k=1}^{2n} z_k \frac{\partial v}{\partial z_k}(0)-\sum_{j,k=1}^{2n} z_jz_k \frac{\partial^2 v}{\partial z_j \partial z_k}(0) -\frac{1}{2}\sum_{j,k=1}^{2n} z_j\bar z_k \left( \frac{\partial^2 v}{\partial z_j \partial \bar z_k}(0)- \sum_{p,q=1}^{2n}J_j^{\bar q}J_{\bar k}^p\frac{\partial^2 v}{\partial z_p \partial \bar z_q}(0) \right).
    \]
    We compute
    \[
    \begin{split}
    \partial \partial_J f&=-\frac{1}{2}\partial \partial_J \sum_{j,k=1}^{2n} z_j\bar z_k \left( \frac{\partial^2 v}{\partial z_j \partial \bar z_k}(0)-\sum_{p,q=1}^{2n} J_j^{\bar q}J_{\bar k}^p\frac{\partial^2 v}{\partial z_p \partial \bar z_q}(0) \right)\\
    &=-\frac{1}{2}\partial J^{-1} \sum_{j,k=1}^{2n} z_j\bar z_k \left( \frac{\partial^2 v}{\partial z_j \partial \bar z_k }(0)- \sum_{p,q=1}^{2n}J_j^{\bar q}J_{\bar k}^p\frac{\partial^2 v}{\partial z_p \partial \bar z_q}(0) \right) d\bar z_k\\
    &=\frac{1}{2}\partial \sum_{j,k,l=1}^{2n} z_j \left( \frac{\partial^2 v}{\partial z_j \partial \bar z_k}(0)- \sum_{p,q=1}^{2n}J_j^{\bar q}J_{\bar k}^p\frac{\partial^2 v}{\partial z_p \partial \bar z_q}(0) \right) J^{\bar k}_{l} d z_l\\
    &= \frac{1}{2}\sum_{j,k,l=1}^{2n}  \left( \frac{\partial^2 v}{\partial z_j \partial \bar z_k}(0)- \sum_{p,q=1}^{2n} J_j^{\bar q}J_{\bar k}^p\frac{\partial^2 v}{\partial z_p \partial \bar z_q}(0) \right) J^{\bar k}_{l} dz_j\wedge d z_l\\
    &= \frac{1}{2}\sum_{j,k,l=1}^{2n} J^{\bar k}_{l}\frac{\partial^2 v}{\partial z_j \partial \bar z_k}(0)dz_j\wedge d z_l+\frac{1}{2} \sum_{j,q,l=1}^{2n}  J_j^{\bar q}\frac{\partial^2 v}{\partial z_l \partial \bar z_q}(0) dz_j\wedge d z_l=0\,.
    \end{split}
    \]
    Similarly, $\partial \partial_J \bar f=0$ and so $h$ is $\partial \partial_J$-closed. Define $ u:=v +h$. It follows that 
    \[
    \eta^*\partial \partial_J u=\eta^* \partial \partial_J v= \Omega\vert_U\,.
    \]
    It remains to show that $0$ is a strict local minimum for $u$ in $\mathbb{B}'$. For any $z=(z_1,\dots,z_{2n})\in \mathbb{B}$ we have
    \[
    \begin{split}
    u(z)- u(0)&=\sum_{k=1}^{2n}z_k \frac{\partial  u}{\partial z_k}(0)+\sum_{k=1}^{2n}\bar z_k \frac{\partial u}{\partial \bar z_k}(0)+\frac{1}{2}\sum_{j,k=1}^{2n}z_jz_k \frac{\partial^2  u}{\partial z_j \partial z_k}(0)\\
    &\quad +\sum_{j,k=1}^{2n}z_j\bar z_k \frac{\partial^2 u}{\partial z_j \partial \bar z_k}(0)+\frac{1}{2}\sum_{j,k=1}^{2n}\bar z_j\bar z_k \frac{\partial^2  u}{\partial \bar z_j \partial \bar z_k}(0)+ o(\|z\|^2)\\
    &=\sum_{j,k=1}^{2n}z_j\bar z_k \frac{\partial^2  u}{\partial z_j \partial \bar z_k}(0)+ o(\|z\|^2)\\
    &=\sum_{j,k=1}^{2n}z_j\bar z_k \left( \frac{\partial^2 v}{\partial z_j \partial \bar z_k}(0)- \frac{1}{2}\frac{\partial^2 v}{\partial z_j \partial \bar z_k}(0)+\frac{1}{2} \sum_{p,q=1}^{2n}J_j^{\bar q}J_{\bar k}^p\frac{\partial^2 v}{\partial z_p \partial \bar z_q}(0) \right)+ o(\|z\|^2) \\
    &=\frac{1}{2}\sum_{j,k=1}^{2n}z_j\bar z_k \left( \frac{\partial^2 v}{\partial z_j \partial \bar z_k}(0)+\sum_{p,q=1}^{2n}J_j^{\bar q}J_{\bar k}^p\frac{\partial^2 v}{\partial z_p \partial \bar z_q}(0) \right)+ o(\|z\|^2)\,.
    \end{split}
    \]
    Since $v$ is strictly quaternionic psh we have $\partial \partial_J v >0$ and a simple calculation shows that
    \[
\mathcal{L}^{\H}(v):=\frac{1}{2}\left(\frac{\partial^2 v}{\partial z_j \partial \bar z_k} + J_j^{\bar q}J_{\bar k}^p \frac{\partial^2 v}{\partial z_p \partial \bar z_q}\right)dz^j\otimes d\bar z^k
\]
is a positive tensor (see for instance \cite[Eqn. 3.5]{Sroka24}). This concludes the proof.
\end{proof}

\begin{proof}[Proof of Theorem $\ref{Thm:continuity}$]
    We argue by contradiction assuming that $\phi$ is not continuous. Equivalently, $d:= \sup(\phi-\phi_*)>0$, where $\phi_*$ is the lower semicontinuous regularisation of $\phi$, i.e.
    \begin{equation}\label{eqn:lsc_reg}
    \phi_*(x):=\sup_{\stackrel{U \text{ open}}{x\in U\subseteq M} } \inf_U \phi\,, \qquad x\in M\,.
    \end{equation}
    Thanks to Theorem \ref{Thm:boundedness} we know that $\phi-\phi_*$ is a bounded upper semicontinuous function. Therefore, by compactness, $\phi-\phi_*$ attains its supremum at some point $x_0\in M$. By definition of $d$ we have that $\phi_*$ attains its minimum on the compact set $\{x\in M \mid \phi(x)-\phi_*(x)=d\}$. Thus $x_0$ may be chosen so that
    \begin{equation}\label{eqn:x_0}
\phi(x_0)=\min_{\{\phi(x)-\phi_*(x)=d\}} \phi\,.
\end{equation}

    By Lemma \ref{Lem:potential} there exist a local chart $\eta \colon U \to \mathbb{B}:=\mathbb{B}_1(0)$ centred at $x_0$, an open ball $\mathbb{B}'\subseteq \mathbb{B}$ and a potential function $u\in C^\infty(\mathbb{B})$ for the HKT form $\Omega$ such that
    \[
    b:=\inf_{\partial \mathbb{B}'} u-u(0)>0\,.
    \]
    Set $\tilde \phi:=\phi \circ \eta^{-1}$. Now, for any $a\in [0,d]$ we define
    \[
    E(a):=\{ x \in \bar{\mathbb{B}}' \mid \tilde \phi(x)-\tilde \phi_*(x)\geq d-a \}\,, \qquad c(a):=\tilde \phi(0)-\inf_{E(a)}\tilde \phi \geq 0\,.
    \]
    We now show that $\lim_{j\to +\infty}c(1/j)=0$.
    Note that the sets $E(a)$ are compact, hence, by lower semicontinuity $\tilde \phi_*$ achieves its infimum over $E(1/j)$ at some $y_j\in E(1/j)$ for all $j\geq 1$. Denote by $y\in M$ the accumulation point of $\{y_j\}$. We must have $y\in E(0)$, indeed
    \[
    \tilde \phi(y)\geq \limsup_{j\to +\infty} \tilde \phi(y_j) \geq \liminf_{j\to +\infty} \tilde \phi(y_j) \geq \liminf_{j\to +\infty} \left(\tilde  \phi_*(y_j)+d-1/j \right) \geq \tilde \phi_*(y)+d\,.
    \]
    Additionally, \eqref{eqn:x_0} tells us that $\tilde \phi_*(y)=\tilde \phi(y)-d\geq \tilde \phi(0)-d= \tilde  \phi_*(0)$ and so
    \[
    \begin{split}
    0\leq \limsup_{j\to +\infty} c(1/j)&=\tilde \phi(0)-\liminf_{j\to +\infty} \inf_{E(1/j)} \tilde \phi \leq \tilde \phi(0)-\liminf_{j\to +\infty }\inf_{E(1/j)} \tilde \phi_*-d\\
    &=\tilde \phi_*(0)-\liminf_{j\to +\infty }\tilde \phi_*(y_j)\leq \tilde \phi_*(0)-\tilde \phi_*(y)\leq 0\,.
    \end{split}
    \]
    Now, since $c(a)$ is non-decreasing there exists $a_0\in (0,d)$ such that
    \begin{equation}\label{eqn:c(a_0)}
    c(a_0)< \frac{1}{3}b\,.
    \end{equation}
    Set $\psi:=u+\tilde \phi$. Since $\phi$ is bounded, we may add a constant to $\psi$ in order to guarantee $\psi>0$ on $\mathbb{B}$ as well as $\psi(0)>d$. Fix $t>1$ such that
    \begin{equation}\label{eqn:a_0}
    t\psi_*(0)+d-\psi(0)=(t-1)(\psi(0)-d)<a_0<(t-1)\left(\psi(0)-d+\frac{2}{3}b\right)\,.
    \end{equation}
    By Proposition \ref{prop_convolution}, up to shrinking $\mathbb{B}$, there exists a sequence $(\psi_j)_{j\in \N} \subseteq \QPSH(\mathbb{B})\cap C^\infty(\mathbb{B}) $ that decreases to $\psi$. Set $w:=t\psi+d-a_0$. By the first inequality of \eqref{eqn:a_0}, there exists another constant $a_1>0$ such that
    \[
    w_*(0)+a_1=t\psi_*(0)+d-a_0+a_1<\psi(0)\leq \psi_j(0)\,, \qquad \text{for all } j\geq 1\,.
    \]
    Hence, by definition of lower semicontinuous regularisation, it follows that the set
    \[
    W(j,c):=\{x\in \bar{\mathbb{B}}' \mid \psi_j(x)-w(x)>c\}
    \]
    is non-empty for all $c\in (0,a_1) $. As a consequence we have
    \begin{equation}\label{eqn:supcontrol}
    \sup_{\mathbb{B}'}(\psi_j-w)>\frac{1}{2}a_1\,, \qquad \text{for all }j\geq 1\,.
    \end{equation}
    On the other hand, there exists an open neighbourhood $V$ of $\partial \mathbb{B}'$ and an index $j_0$ such that
    \begin{equation}
        \label{eqn:nhoodboundary}
    \psi_j-w<0\,,
    \end{equation}
    on $V$ for all $j\geq j_0$. To prove this, choose a point $z\in \partial \mathbb{B}'\cap E(a_0)$. From \eqref{eqn:c(a_0)} we deduce
    \[
    \psi_*(z)\geq \inf_{\partial \mathbb{B}'} u+ \tilde \phi_*(z)\geq u(0)+b+\tilde \phi(z)-d\geq  u(0)+b+\tilde \phi(0)-c(a_0)-d \geq \psi(0)-d+\frac{2}{3}b\,.
    \]
    Then, from the second inequality in \eqref{eqn:a_0} we have $(t-1)\psi_*(z) >a_0$ and consequently
    \[
    \psi(z)\leq \psi_*(z)+d<t\psi_*(z)+d-a_0=w_*(z)\,.
    \]
    Lemma \ref{Cor:Hartogs} now implies the existence of an open neighbourhood $V_1$ of $\partial\mathbb{B}' \cap E(a_0)$ and an index $j_1$ such that \eqref{eqn:nhoodboundary} holds on $V_1$ for all $j\geq j_1$. Furthermore, by construction on $\partial \mathbb{B}'\setminus( \partial \mathbb{B}'\cap V_1)$
    \[
    \psi-t\psi_*\leq \psi-\psi_*<d-a_0.
    \]
    In particular, another application of Lemma \ref{Cor:Hartogs}, this time on $\partial \mathbb{B}'\setminus( \partial \mathbb{B}'\cap V_1)$ implies that there exists an open neighbourhood $V_2$ of $\partial \mathbb{B}'\setminus( \partial \mathbb{B}'\cap V_1) $ over which \eqref{eqn:nhoodboundary} holds for all $j\geq j_2$ for some index $j_2$. Setting $V:=V_1\cup V_2$ and $j_0:=\max(j_1,j_2)$ we have that $\eqref{eqn:nhoodboundary}$ holds on $V$ for all $j\geq j_0$, as claimed.
    This implies that there exists a smaller ball $\mathbb{B}''\Subset \mathbb{B}'$ such that $  W(j,c) \subset \bar{\mathbb{B}}'' $ whenever $j\geq j_0$.
    
    Now, since $\phi$ solves \eqref{eq_QMA}, and $\psi_j-w<0$ on $V$ for all $j\geq j_0$, we see that the functions $w$ and $\psi_j$ satisfy the hypotheses of Proposition \ref{Prop:supCap} on $\mathbb{B}'$. We deduce that there exist constants $C_j$ such that
    \begin{equation}\label{eqn:AI_will_destroy_us}
        \sup_{\mathbb{B}'}(\psi_j-w)\leq \frac{a_1}{4} + C_j\cdot \mathrm{Cap^{WK}}\bigl(W(j,a_1/4),\mathbb{B}'\bigr)^{\alpha/n}.
    \end{equation}
    We then follow the proof of Proposition \ref{Prop:supCap} to show that the constants $C_j$ are uniform in $j\in\N$. Indeed, as $(\psi_j)_{j\in \N}$ is decreasing, the family of functions $f_j(s):=\mathrm{Cap^{WK}}\bigl(W(j,s+a_1/4),\mathbb{B}'\bigr)^{1/n}$ is decreasing in $j$. Thus, for any $\alpha, B$ such that \eqref{eq_boundedness} holds for $f_j$ for any $j\in\N$, we get a uniform $S>0$ such that $\sup_{\mathbb{B}'}(\psi_j-w)\leq S+\frac{a_1}{4}$. Therefore, since the constants $B, \alpha$ appearing in the proof of Proposition \ref{Prop:supCap} do not depend on $j$, we conclude that the $C_j$'s can be taken to be uniform in $j$. 

    Hence, recalling \eqref{eqn:supcontrol}, from \eqref{eqn:AI_will_destroy_us} we deduce that
    \begin{equation}\label{eq:capbound}
    C\cdot \mathrm{Cap^{WK}}\bigl(W(j,a_1/4),\mathbb{B}'\bigr)^{\alpha/n}>\frac{1}{4}a_1>0\,, \qquad \text{for all }j\geq j_0\,.
    \end{equation}
    On the other hand, since
    \[
    W(j,a_1/4)\subset \left\{ x\in \bar{\mathbb{B}}' \mid \psi_j(x)-\psi(x)-d+a_0> \frac{1}{4}a_1 \right\}\cap \bar{\mathbb{B}}''\,, \qquad \text{for all }j\geq 1\,,
    \]
    using the monotonicity of the capacity, we claim that
    \begin{equation}\label{eq:capconv}
    \lim_{j\to +\infty} \mathrm{Cap ^{WK}}\bigl(W(j,a_1/4),\mathbb{B}'\bigr)=0\,.
    \end{equation}
    To prove this, we set $\epsilon:=d-a_0+\frac{1}{4}a_1>0$ and note that
    \[
    \begin{split}
    \mathrm{Cap^{WK}}\bigl(\bar{\mathbb{B}}'' \cap \{\psi_j-\psi > \epsilon\},\mathbb{B}' \bigr)&= \sup\left\{ \int_{\bar{\mathbb{B}}'' \cap \{\psi_j-\psi > \epsilon\} } (\partial \partial_J h)^n \mid h\in \QPSH(\mathbb{B}'), \, 0\leq h\leq 1 \right\}\\
    & \leq \sup\left\{ \epsilon^{-1} \int_{\bar{\mathbb{B}}''} (\psi_j-\psi)(\partial \partial_J h)^n \mid h \in \QPSH(\mathbb{B}'), \, 0\leq h\leq 1 \right\}\,.
    \end{split}
    \]
    Taking the limit for $j\to +\infty$ we deduce \eqref{eq:capconv}:
    \[
    0\leq \lim_{j\to +\infty} \mathrm{Cap ^{WK}}\bigl(W(j,a_1/4),\mathbb{B}'\bigr)\leq \lim_{j\to +\infty} \mathrm{Cap^{WK}}\bigl(\bar{\mathbb{B}}''\cap \{\psi_j-\psi > \epsilon\},\mathbb{B}' \bigr)=0\,,
    \]
    thanks to \cite[Lem. 3.2]{Wan-Zhang}. The contradiction now follows from \eqref{eq:capbound} and \eqref{eq:capconv}, thus concluding the proof.
    \end{proof}

\section{Uniqueness of solutions}\label{Sec 8}
The purpose of the present section is to prove that solutions to the quaternionic Monge--Amp\`ere equation are unique, namely:
\begin{thm}\label{thm:Uniqueness}
Let $(M^n,I,J,K,\Omega)$ be a compact locally flat HKT manifold such that $K_I(M)$ is holomorphically trivial and let $ \phi, \psi \in \QPSH(M,\Omega)$ be solutions to \eqref{eq_QMA}. Then $\phi=\psi$.
\end{thm}

Uniqueness in the classical K\"ahler case was established through a series of successive improvements \cite{Calabi,Blocki,Guedj-Zeriahi 2007,Dinew}. Since in our main theorem, the datum $F$ is assumed to be smooth, it is enough to adapt the argument in \cite{Blocki} to the quaternionic case.

Here we need a Cauchy-Schwarz type inequality on locally flat HKT manifolds. If $S$ is a q-closed $(2n-2,0)$-current, we can define the $\partial_J$-closed $(2n-1,0)$-current $\partial_J \psi\wedge S:= \partial_J(\psi S)$ for any $\psi\in \QPSH(M,\Omega)\cap L^\infty(M)$ (cf. section \ref{sec 3}). 
    Then, for any $\varphi\in \QPSH(M, \Omega)\cap L^\infty(M)$, the $(2n,0)$-current $\partial \varphi\wedge \partial_J \psi\wedge S$ can be naturally defined as
    $$
    \partial \varphi \wedge \partial_J \psi \wedge S:= \partial (\varphi \partial_J\psi\wedge S ) - \varphi \partial\left(\partial_J \psi \wedge S\right)=\partial (\varphi \partial_J\psi\wedge S ) - \varphi \partial\partial_J \psi \wedge S\,.
    $$

\begin{lem}\label{cor:CS}
Let $T$ be a q-closed q-positive $(2n-2,0)$-current. If $\phi,\psi$ are linear combinations of bounded $\Omega$-quaternionic psh functions, then the following Cauchy-Schwarz type inequality holds:
    $$
    \left\lvert \int_M \partial \phi\wedge \partial_J \psi \wedge T\right\rvert\leq \left(\int_M \partial \phi \wedge \partial_J \phi  \wedge T \right)^{1/2}\left( \int_M \partial \psi \wedge \partial_J \psi \wedge T \right)^{1/2}\,.
    $$
\end{lem}
\begin{proof}
It immediately follows from Lemma \ref{lem:fakegradient} as the form $\Phi(\varphi,\psi):=\int_M\partial \varphi\wedge \partial_J \psi \wedge T$ is $\R$-bilinear, symmetric and positive semidefinite.
\end{proof}

\begin{proof}[Proof of Theorem $\ref{thm:Uniqueness}$]
We first observe that $\phi,\psi\in \mathcal{E}^1(M,\Omega)$. Indeed $\phi\in L^1(\MA_\phi)=L^1(\mathrm{e}^F\Omega^n)$ since $\phi$ is $\Omega$-quaternionic psh, and similarly for $\psi$. Moreover, replacing $\varphi, \psi$ by their translations $\varphi-\sup_M \varphi, \psi-\sup_M \psi$, we can assume that $\varphi, \psi\leq 0 $ as it suffices to prove that $\varphi-\psi$ is constant.

Note that, by Theorem \ref{Thm:boundedness}, $\phi,\psi$ are globally bounded. Set $f=(\phi-\psi)/2,$ $ h=(\phi+\psi)/2$. By Proposition \ref{prop:cptness} we have $h\in \mathcal{E}^1(M,\Omega)$, so that we can assume without loss of generality that  $\int_M(-h)\Omega_h^n\geq 1$. Letting $T$ be a q-closed $(2n-4,0)$-current we have
    $$
    \int_M \partial f \wedge \partial_J h \wedge \partial \partial_J f \wedge T=\int_M \partial f \wedge \partial_J f \wedge \partial \partial_J h \wedge T 
    $$
    as both sides coincide with $-\int_M f \partial \partial_J h \wedge \partial \partial_J f \wedge T$. We deduce that 
    \begin{align*}
        \int_M \partial f\wedge \partial_J f\wedge \Omega \wedge T&=\int_M \partial f\wedge \partial_J f\wedge \Omega_h \wedge T- \int_M \partial f\wedge \partial_J f\wedge \partial \partial_J h \wedge T\\
        &=\int_M \partial f\wedge \partial_J f\wedge \Omega_h \wedge T- \int_M \partial f\wedge \partial_J h \wedge \partial\partial_J f  \wedge T\\
        &=\int_M \partial f\wedge \partial_J f\wedge \Omega_h \wedge T- \int_M \partial f\wedge \partial_J h \wedge \left(\frac{\Omega_\phi-\Omega_\psi}{2}\right)  \wedge T\,.
    \end{align*}
    By Lemma \ref{cor:CS}, since $\Omega_\phi\leq 2 \Omega_h$, we have 
    \[
    \begin{split}
    \left\lvert \int_M \partial f\wedge \partial_J h \wedge \Omega_\phi  \wedge T\right\rvert&\leq 2  \left(\int_M \partial f\wedge \partial_J f \wedge \Omega_h  \wedge T \right)^{1/2}\left( \int_M \partial h\wedge \partial_J h \wedge \Omega_h  \wedge T \right)^{1/2}\\
    &\leq 2  \left(\int_M \partial f\wedge \partial_J f \wedge \Omega_h  \wedge T \right)^{1/2}\left( \int_M (-h)\Omega_h^2  \wedge T \right)^{1/2}\,.
    \end{split}
    \]
    Similarly we can control $\int_M \partial f\wedge \partial_J h \wedge \Omega_\psi  \wedge T$. Moreover, by integrating by parts we infer 
    $$
    \int_M \partial f \wedge \partial_J f \wedge \Omega_h\wedge T= \frac{1}{4}\int_M (\psi-\phi)\left(\Omega_\phi-\Omega_\psi\right)\wedge \Omega_h\wedge T\leq \int_M (-h)\Omega_h^2\wedge T \,,
    $$
    where we used $(-h)\Omega_h+\frac{1}{4}(\phi-\psi)(\Omega_\phi-\Omega_\psi)=-\frac{1}{2}\left(\psi\Omega_\phi+\phi \Omega_\psi\right)\geq 0$. Hence
    \begin{equation}\label{ineq1}
        \int_M \partial f \wedge \partial_J f \wedge \Omega \wedge T
        \leq  5\left(\int_M \partial f\wedge \partial_J f\wedge \Omega_h \wedge T\right)^{1/2}\left(\int_M(-h)\Omega_h^2\wedge T\right)^{1/2}\,.
    \end{equation}
    For $T=\Omega_h^l\wedge \Omega^{n-2-l}$, where $l=0,\dots,n-2$, we also have that
    \begin{equation}\label{ineq2}
    \int_M(-h)\Omega_h^2\wedge T\leq\int_M(-h)\Omega_h^n\,.
    \end{equation}
    Thus, applying iteratively the inequalities \eqref{ineq1}, \eqref{ineq2} to $T=\Omega_h^l \wedge \Omega^{n-2-l}$ for all $l=0,\dots,n-2$ we deduce
    \[
    \begin{split}
        \int_M \partial f\wedge \partial_J f \wedge \Omega^{n-1}&\leq 5^{\sum_{j=0}^{n-2} \frac{1}{2^j}}\left(\int_M \partial f\wedge \partial_J f \wedge \Omega_h^{n-1}\right)^{\frac{1}{2^{n-1}}}\left(\int_M(-h)\Omega_h^n\right)^{\sum_{j=1}^{n-1}\frac{1}{2^j}}\\
        &\leq 25 \left(\int_M \partial f\wedge \partial_J f \wedge \Omega_h^{n-1}\right)^{\frac{1}{2^{n-1}}}\int_M(-h)\Omega_h^n\,,
    \end{split}
    \]
    where we also used that $\int_M (-h)\Omega_h^n\geq 1$. But, since $\binom{n-1}{l}\leq 2^{n-1}$,
    $$
    \int_M \partial f \wedge \partial_J f \wedge \Omega_h^{n-1}\leq \int_M \partial f \wedge \partial_J f \wedge \sum_{l=0}^{n-1}\Omega_\phi^l \wedge \Omega_\psi^{n-1-l}=\int_M \frac{-f}{2}\left(\Omega_\phi^n-\Omega_\psi^n\right)=0\,.
    $$
    Hence we have
    $$
    \int_M \partial f \wedge \partial_J f\wedge \Omega^{n-1}=0\,,
    $$
    i.e. $\partial f=0$, which concludes the proof.
\end{proof}

\appendix
\section{Technical Results}\label{sec:Appendix}
In this appendix, we collect some technical results used in the paper. Although some of them are already known in the literature, we provide proofs for the reader’s convenience.    

\subsection{Approximations of quaternionic psh functions.}
In the classical K\"ahler setting, $\omega$-psh functions on compact manifolds can be approximated by smooth $\omega$-psh functions; see \cite[Thm.~1]{Blocki-Kolodziej}. To the best of our knowledge, no analogous result is currently known for $\Omega$-quaternionic psh functions on compact HKT manifolds. The main technical obstruction to adapting the argument of \cite{Blocki-Kolodziej} to the quaternionic setting appears to be, as is often the case, the lack of a quaternionic Liouville-type theorem.

In this section, we prove that every bounded $\Omega$-quaternionic psh function can be approximated by smooth functions that are plurisubharmonic with respect to a suitable multiple of $\Omega$. As a first step we consider the local case: 

\begin{prop}[{\cite[Prop. 2.1.(5)]{Wan 2020}}]\label{prop_convolution}
Let $A\subseteq \H^n$ be a domain, $ \phi\in \QPSH(A) $ and denote $\phi_\epsilon=\phi\star \chi_\epsilon $ the standard regularisation on $ A_\epsilon=\{ x\in A \mid \mathrm{dist}(x,\partial A)>\epsilon \} $ via mollification. Then $ \phi_\epsilon\in \QPSH(A_\epsilon)\cap C^\infty(A_\epsilon) $ and it decreases to $ \phi $ as $ \epsilon \searrow 0 $.
\end{prop}
\begin{proof}
    As $\varphi$ is subharmonic, the sequence $\varphi_\varepsilon$ is subharmonic and decreases to $\varphi$ as $\varepsilon\searrow 0$. To conclude the proof it then remains to show that $\varphi_\varepsilon$ satisfies the submean inequality when restricted on an affine right quaternionic line. Namely, we want
    $$
    \phi_\epsilon(x)\leq \frac{1}{\lvert \mathbb{B}\rvert}\int_{\mathbb{B}} \phi_\epsilon(x+vy) \, d\lambda_1(y)
    $$
    for any $x\in A_\epsilon$ and any $v\in \H^n$ such that $x+v y\in A_\epsilon$ for any $y\in\mathbb{B}$, where $\mathbb{B}$ is the $4$-dimensional unit ball and $d\lambda_r$ is the Lebesgue measure on $\R^{4r}$. As $\phi$ is quaternionic psh, identifying $\H^n\simeq \R^{4n}$ it follows that 
    $$
    \varphi_\epsilon(x)= \int_{\R^{4n}}\varphi(x-a)\chi_\epsilon(a)\,d\lambda_n(a)\leq \frac{1}{\lvert \mathbb{B} \rvert}\int_{\R^{4n}} \int_{\mathbb{B}}\varphi(x-a+vy) \chi_\epsilon(a) \,d\lambda_1(y) d\lambda_n(a)
    $$
    for any $v\in\H^n$ such that $\lVert v\rVert\leq \mathrm{dist}(x,\partial A)- \epsilon$ so that $x-a+vy\in A$ for any $y\in \mathbb{B}$ and for any $a\in \R^{4n}$ such that $\|a\|\leq \epsilon$. Switching the integrals we deduce that
    $$
    \varphi_\epsilon(x)\leq \frac{1}{\lvert \mathbb{B}\rvert}\int_{\mathbb{B}} \int_{\R^{4n}}\varphi(x-a+vy) \chi_\epsilon(a) \,d\lambda_n(a) d\lambda_1(y)=  \frac{1}{\lvert \mathbb{B}\rvert}\int_{\mathbb{B}} \varphi_\epsilon(x+vy) \, d\lambda_1(y)\,,
    $$
    concluding the proof.
\end{proof}
Now we focus on the main result of the present subsection. 
\begin{prop}\label{Prop:approx}
    Let $\varphi\in \QPSH(M,\Omega)\cap L^\infty(M)$, then there exists a constant $C>0$, depending only on $(M, \Omega),\|\phi\|_{L^\infty(M)}$, and a family $\{\varphi_\epsilon\}_{\varepsilon>0}\subset\QPSH(M,C\Omega)\cap C^{\infty}(M)$ such that $\varphi_{\epsilon}\searrow \varphi$.
\end{prop}
\begin{proof}

   Let $\{B_j\}_{j=1}^N$ be a finite cover on $M$ such that $\Omega\vert_{B_j}=\partial \partial_J u_j$ for some $u_j\in C^\infty(B_j)$ and the norms $\|u_j\|_{L^{\infty}(B_j)}$  are uniformly bounded in $j$. Let also assume that there exist $B_j'\Subset B_j$ such that $\{B_j'\}_{j=1}^N$ is still a cover of $M$. Let also $\{\rho_j\}$ be a partition of unity subordinated to the cover $\{B_j'\}$. Define $\phi_{j,\epsilon}:= (\phi \vert_{B_j}+u_j)\star\chi_\epsilon -u_j$. By Proposition \ref{prop_convolution}, for $\epsilon>0$ small enough, $\phi_{j,\epsilon}\in \QPSH(B_j',\Omega)\cap C^\infty(B_j')$ and $\phi_{j,\epsilon}$ is uniformly bounded in $j,\epsilon$ on $B_j'$. 
   We then set
    $$
    \phi_\varepsilon:=\log\left(\sum_j \rho_j e^{\phi_{j,\varepsilon}}\right)\,.
    $$
    We claim that there exists a constant $C>0$ such that $\phi_\varepsilon \in \QPSH(M,C\Omega)$.
    
    Without loss of generality we can assume that $\rho_k=\theta_k^2$ for $\theta_k$ smooth function, and we set $\alpha_k:=2\partial \theta_k +\theta_k \partial \phi_{k,\epsilon}$.
    Then, using 
    \[
    J^{-1}\overline{\alpha}_k=2\partial_J \theta_k+\theta_k \partial_J \phi_{k,\epsilon}\,,
    \]
    we obtain
    \[
    \partial_J \phi_\epsilon= \frac{\sum_k (2\theta_k\partial_J \theta_k +\theta_k^2 \partial_J \phi_{k,\epsilon})\mathrm{e}^{\phi_{k,\epsilon}}}{\sum_k \theta_k^2 \mathrm{e}^{\phi_{k,\epsilon}}}=\frac{\sum_k \theta_kJ^{-1}\overline{ \alpha}_k\mathrm{e}^{\phi_{k,\epsilon}}}{\sum_k \theta^2_k \mathrm{e}^{\phi_{k,\epsilon}}}
    \]
    and 
    \[
    \begin{split}
    \partial \partial_J \phi_\epsilon
    &= \frac{\sum_k(-2\partial \theta_k \wedge \partial_J \theta_k +2\theta_k\partial \partial_J \theta_k+\theta_k^2 \partial \partial_J \phi_{k,\epsilon} + \alpha_k \wedge J^{-1}\overline{\alpha}_k)\mathrm{e}^{\phi_{k,\epsilon}}}{\sum_k \theta_k^2 \mathrm{e}^{\phi_{k,\epsilon}}}\\
    &\qquad -\frac{\sum_{j,k}\theta_j\theta_k\alpha_j \wedge J^{-1}\overline{\alpha}_k\mathrm{e}^{\phi_{j,\epsilon}}\mathrm{e}^{\phi_{k,\epsilon}}}{(\sum_k \theta_k^2 \mathrm{e}^{\phi_{k,\epsilon}})^2}\,.
    \end{split}
    \]
    Since the $\rho_k$'s are smooth, we have $2\theta_k\partial \partial_J \theta_k-2\partial \theta_k \wedge \partial_J \theta_k\geq -C_1\Omega$ for a constant $C_1>0$ uniform in $k$. Also, we have $\partial \partial_J \phi_{k,\epsilon} \geq - \Omega $. Furthermore the Legendre identity  
    $$
    \left(\sum_j a_j^2\right)\left(\sum_k b_k^2\right)- \sum_{j,k}a_j b_ja_kb_k=\sum_{j<k} \lvert a_jb_k-a_kb_j\rvert^2
    $$
    applied to $(1,0)$-forms gives
    $$
    \left(\sum_j\alpha_j\wedge J^{-1}\overline{\alpha}_je^{\phi_{j,\epsilon}}\right)\left(\sum_k \theta_k^2e^{\phi_{k,\epsilon}}\right)- \sum_{j,k} \theta_je^{\phi_{j,\epsilon}}\alpha_j \wedge \theta_k e^{\phi_{k,\epsilon}}J^{-1}\overline{\alpha}_k=\sum_{j<k} \sigma_{j,k}\wedge J^{-1}\overline{\sigma}_{j,k} \geq 0\,,
    $$
    where $\sigma_{j,k}=\theta_ke^{\phi_{k,\epsilon}/2}\alpha_je^{\phi_{j,\epsilon}/2}-\theta_je^{\phi_{j,\epsilon}/2}\alpha_ke^{\phi_{k,\epsilon}/2}$. 
    It follows
    $$
    \partial\partial_J \phi_\epsilon\geq -\frac{1+C_1\sum_k e^{\varphi_{k,\epsilon}}}{\sum_k \theta_k^2 e^{\varphi_{k,\epsilon}}}\Omega \geq  
    -C \Omega
    $$
    for a constant $C$ which does not depend on $\epsilon$, where in the last inequality we used that $\phi_{k,\epsilon}$ is uniformly bounded in $\epsilon$ and $k$.
\end{proof}

\subsection{Sequences of quaternionic psh functions}
In this subsection we focus on sequences of quaternionic psh functions.  
\begin{prop}[{\cite[Prop. 2.1]{Wan 2020}}]\label{3}
Let $A\subseteq \H^n$ be a domain. The following properties hold:
\begin{enumerate}
\item If $ (\phi_j)_{j\in \N} \subseteq \QPSH(A) $ is a decreasing sequence then $ \lim_j \phi_j $ is either quaternionic plurisubharmonic or constant $ -\infty $.

\vspace{0.1cm}
\item If $ (\phi_j)_{j\in J}\subseteq \QPSH(A) $ is a family such that $ \phi=\sup_{j\in J}\phi_j $ is locally bounded from above, then also the upper semicontinuous regularisation $ \phi^*(x)=\limsup_{A\ni y\to x}\phi(y) $ belongs to $\QPSH(A)$.
\end{enumerate}
\end{prop}
\begin{proof}
    We recall that subharmonic functions on an open set of $\R^{4n}$ satisfy the properties of the Proposition. Thus (1) directly follows by restriction to affine right quaternionic lines.

    Let us now prove (2). By Choquet's Lemma \cite[Lem. 4.31]{Guedj-Zeriahi book} there exists a countable subfamily $\{j_k\}_k $ such that $\varphi=\sup_{k\in\N}\varphi_{j_k}$. Set $\psi_l:=\max_{k=1,\dots,l}\varphi_{j_k}$. By restriction to affine right quaternionic lines, we obtain that $\psi_l\in \QPSH(A)$. Moreover $\psi_l$ is now increasing in $l\in\N$ and $\varphi=\sup_l \psi_l$. Let $\psi_{l,\varepsilon}:=\psi_l\star\chi_\varepsilon $ be the standard regularisation via mollification. By Proposition \ref{prop_convolution} we know that $\psi_{l,\varepsilon}\in \QPSH(A)\cap C^\infty(A)$ is such that $(l,\varepsilon)\to \psi_{l,\varepsilon}$ is increasing in $l$ and in $\varepsilon$. In particular the function $\psi_{\varepsilon}:=\lim_{l\to +\infty}\psi_{l,\varepsilon}$ is continuous, locally bounded from above and it satisfies the submean inequality when restricted to any affine right quaternionic line, i.e. $\psi_{\varepsilon}$ is quaternionic psh. By construction we have
    $$
    \varphi\leq \varphi^*\leq \psi_{\varepsilon}\leq \varphi\star \chi_\varepsilon\,.
    $$
    Since $\varphi\in L^1_{\rm loc}$ the sequence $\varphi\star \chi_\varepsilon$ converges to $\varphi$ in $L^1_{\rm loc}$ as $\varepsilon\to 0$, while $\psi_\varepsilon$ decreases to a quaternionic psh function $\psi$ by $(1)$. Hence $\varphi^*=\psi\in \QPSH(A)$ and (2) follows.      
\end{proof}


Next, we observe how some local properties of quaternionic psh functions can be generalised to $\Omega$-quaternionic psh functions on compact locally flat HKT manifolds.

\begin{prop}\label{prop_basic}
Let $(M,I,J,K,\Omega)$ be a compact locally flat HKT manifold:
\begin{enumerate}

\item If $(\phi_j)_{j\in \N} $ is a decreasing sequence in $\QPSH(M,\Omega)$ then either it converges uniformly to $-\infty$ or $ \lim_j \phi_j \in \QPSH(M,\Omega)$.

\vspace{0.1cm}
\item If $ (\phi_j)_{j\in \N} $ is a sequence in $\QPSH(M,\Omega)$ such that $ \varphi=\sup_{j\in \N}\varphi_j $ is bounded from above, then the upper semicontinuous regularisation $ \varphi^* $ belongs to $ \QPSH(M,\Omega) $.
\vspace{0.1cm}
\item If $(\phi_j)_{j\in \N} $ is a sequence in $\QPSH(M,\Omega)$ convergent to $\phi\in \QPSH(M,\Omega)$ then the sequence $\psi_j:=(\sup_{l\geq j}\phi_l)^*$ decreases to $\phi$.
\item If $(\phi_j)_{j\in \N}$ is a sequence in $\QPSH(M,\Omega)$ which is uniformly bounded from above and which does not converge uniformly to $-\infty$, then it admits a subsequence $(\phi_{j_k})_{k\in\N} $ that converges to $\phi\in \QPSH(M,\Omega)$ in $L^1(M)$.
\end{enumerate}
\end{prop}
\begin{proof}
    The first three points are immediate consequences of Proposition \ref{3} since on any local quaternionic chart such that $\Omega=\partial \partial_J u$ the functions $u+\phi_j$ are quaternionic psh.

    Also for the last point it is enough to work locally. So let us assume that $\phi_j\in \QPSH(A)$, where $A$ is a domain of $\H^n$ and suppose that they do not converge locally uniformly to $-\infty$.
    Without loss of generality we may suppose $\phi_j \leq 0$. Since $(\phi_j)_{j\in \N}$ does not converge to $-\infty$ locally uniformly, there exist a constant $C>0$ and a compact subset $K\subseteq A$ such that
    $\limsup_{j\to +\infty} \sup_K \phi_j >-C$. Up to a subsequence, there exist points $x_j\in K$ such that $\phi_j(x_j) \geq -2C$ and $x_j\to x\in K$. Fix a ball $\mathbb{B}\subseteq A$ centred at $x$. For $j$ large enough there exists a ball $\mathbb{B}_j$ centred at $x_j$ and such that $\mathbb{B}\subseteq \mathbb{B}_j \subseteq A$. Since quaternionic psh functions are subharmonic, it follows that
\[
\int_{\mathbb{B}} \phi_j d\lambda \geq \int_{\mathbb{B}_j} \phi_j d\lambda \geq \lambda(\mathbb{B}_j)\phi_j(x_j) \geq -2C\lambda(\mathbb{B}_j)\,,
\]
where $\lambda$ is the Lebesgue measure on $A$.

Next, consider the set $X$ of points in $A$ that have a neighbourhood $U\subseteq A$ such that $\int_U \phi_jd\lambda $ is uniformly bounded from below. Clearly $X$ is open and non-empty by what was proved above, we aim to show that it is also closed. Let $y$ be a point in the closure of $X$ and fix a ball $\mathbb{B}_\epsilon(y)$ centred at $y$ of radius $\epsilon>0$ small enough so that $\mathbb{B}_\epsilon(y)\subseteq A$. There exists a point $z\in \mathbb{B}_\epsilon(y)\cap X$ and thus $z$ has a neighbourhood $U\subseteq A$ such that $\int_U \phi_jd\lambda$ is uniformly bounded from below. In particular, there exist points $z'_j\in \mathbb{B}_\epsilon(y)$ close to $z$ such that $\phi_j(z'_j)>-C'$ for all $j$ and we infer that $\int_{\mathbb{B}_{\epsilon+\sigma}(z'_j)}\phi_jd\lambda$ is uniformly bounded for any $\sigma>0$ small. Choosing the points $z_j$ sufficiently close to $z$ we have that $\mathbb{B}_{\epsilon+\sigma}(z_j')\supset \mathbb{B}_{\epsilon}(z)$ for any $j\in\N$. We deduce that $\mathbb{B}_{\varepsilon}(z)$ is a neighborhood of $y$ and that $\int_{\mathbb{B}_{\varepsilon}(z)}\varphi_jd\lambda\geq \int_{\mathbb{B}_{\varepsilon+\sigma}(z_j)}\varphi_jd\lambda$ is uniformly bounded in $j\in\N$.
We conclude that $y\in X$ and thus $X$ is closed. Consequently, $X=A$ and the sequence $\phi_j$ is bounded in $L^1_{\mathrm{loc}}(A)$.
%
%

We can now proceed similarly to the complex setting (cf. \cite[Thm. 1.46.(2)]{Guedj-Zeriahi book}).
Indeed the sequence of non-negative measures $\mu_j:=(-\phi_j)\lambda$ is bounded in the weak topology of Radon measures on $A$ and thus admits a convergent subsequence. By the classical theory of distribution the positive limit Radon measure $\mu$ can be identified with a distribution that we will denote $-\psi$. For any non-negative $\alpha\in \Lambda^{2n-2,2n}(A)$ with compact support we then have
\begin{equation}
    \label{eqn:2}
    \begin{split}
        0\leq (\partial \partial_J \phi_j)(\alpha)&= \int_A \phi_j \partial \partial_J \alpha= \int_A \phi_j f_\alpha \, d\lambda
        \\
        &=-\int_A f_\alpha\, d\mu_j \longrightarrow -\int_A f_\alpha \, d\mu= \psi(f_\alpha)=(\partial \partial_J\psi)(\alpha)\,,
    \end{split}
\end{equation}
where we clearly set $\partial\partial_J \alpha=f_\alpha d\lambda$ for $f_\alpha\geq 0$ and where we recall that
$$
(\partial \partial_J \psi)(\alpha):=\psi(f_\alpha)\,,
$$
by the usual extension of derivation to distributions (as in \eqref{eqn:curr}) and the duality between top forms and functions given by the Lebesgue measure, i.e. $ \psi(f\, d\lambda):= \psi(f). $
Letting $\chi_\epsilon$ be the standard mollifiers as in Proposition \ref{prop_convolution}, we set $\psi_\epsilon(y):=\psi(\chi_{\epsilon,y})$ for $y\in A_\epsilon, $ where $\chi_{\epsilon,y}(x):=\chi_\epsilon(y-x)$ for $x\in \R^{4n}$. Then
\[
\begin{split}
    (\partial \partial_J \psi_\epsilon)(\alpha)&=\int_{A_\epsilon} \psi_\epsilon (y) f_\alpha(y)\,d\lambda(y)=\int_{A_\epsilon} \psi(\chi_{\epsilon,y})f_\alpha(y) \, d\lambda(y)\\
    &=\int_{A_\epsilon} \left(-\int_{A} \chi_{\epsilon,y}(x)\, d\mu(x)\right) f_\alpha(y)\, d\lambda(y)=-\int_{A} \left(\int_{A_\epsilon} \chi_{\epsilon}(y-x)f_\alpha(y)\, d\lambda(y) \right) \, d\mu(x)\\
    &=-\int_A \left(\int_{\mathbb{B}_\epsilon}\chi_\epsilon(y)f_\alpha(y+x)\,d\lambda(y)\right)\, d\mu(x)\,,
\end{split}
\]
where in the last equality we made a change of variables, we extended $f_\alpha=0$ outside $A$, observing that for $\epsilon>0$ small enough $d(\mathrm{Supp}(f_\alpha), \partial A)>\epsilon$, and we set $\mathbb{B}_\epsilon\subset \R^{4n}$ for the ball of radius $\epsilon$. Note that we also used the translation invariance of the Lebesgue measure. We then set $(\tau_yf_\alpha)(x):=f_\alpha(x+y)$ and similarly $(\tau_y\alpha)(x):=\alpha(x+y)$, and observe that $\tau_y\alpha$ (and $\tau_yf_\alpha$) has compact support in $A$ for any $y\in \mathbb{B}_\epsilon$ if $\epsilon>0$ is small enough. It follows that
\begin{align*}
    (\partial\partial_J\psi_\epsilon)(\alpha)&= \int_{\mathbb{B}_\epsilon}\chi_{\epsilon}(y)\left(-\int_A  f_{\alpha}(y+x)d\mu(x)\right) \, d\lambda(y)\\
    &= \int_{A_\epsilon}\chi_\epsilon(y) \psi(\tau_yf_\alpha) d\lambda(y)= \int_{A_\epsilon}\chi_{\epsilon}(y)(\partial \partial_J \psi)(\tau_y\alpha) d\lambda(y)\,.
\end{align*}
Since $\tau_y\alpha$ is q-positive if $\alpha$ is q-positive, from \eqref{eqn:2} we obtain that $\partial \partial_J \psi_\epsilon\geq 0$ as a current. To conclude the proof we can now proceed as in Proposition \ref{prop:1}. Namely, $\psi_\epsilon\in \QPSH(A_\epsilon) $ by \cite[Prop. 2.1.6]{Alesker 2003},  and we claim that the sequence $(\psi_\epsilon)_\epsilon$ is non-decreasing in $\epsilon$. Since $\psi_\epsilon $ is subharmonic, the convolution $\psi_\epsilon \star \chi_\tau $ is non-decreasing in $\tau$. It follows that $\psi_\epsilon \star \chi_\tau=\psi_\tau \star \chi_\epsilon $ is non-decreasing in $\epsilon$ and taking the limit as $\tau \to 0$ we deduce that $\psi_\epsilon$ must be non-decreasing as well (cf. Proposition \ref{prop_convolution}). Let $\phi$ be the limit of $\psi_\epsilon$ as $\epsilon \searrow 0$ and note that $\phi=\psi$ as distributions, so that $\phi$ cannot be identically $-\infty$. Thanks to Proposition \ref{3} we then deduce that $\phi\in \QPSH(A)$. In particular $\phi_j$ converges to $\phi$ weakly, i.e. against continuous functions with compact support in $A$.

In order to conclude the proof we need to show that $\phi_j \to \phi$ in $L^1_{\mathrm{loc}}(A)$. Fix a compact subset $K\subset A$, and a continuous function $\eta$ such that $\eta\vert_K=1$ and $0\leq \eta \leq 1$ on $A$, then we have
\begin{equation}\label{inH}
\int_{K} |\phi_j-\phi| \, d \lambda \leq  \int_{A} \eta (\phi_j\star \chi_\epsilon-\phi_j) \, d \lambda  + \int_{A} \eta |\phi_j\star \chi_\epsilon-\phi\star \chi_\epsilon | \, d \lambda + \int_{A} \eta (\phi \star \chi_\epsilon -\phi) \, d \lambda \,,
\end{equation}
where we used that $\phi_j \star \chi_\epsilon \geq \phi_j$ and $\phi \star \chi_\epsilon \geq \phi$ (see again Proposition \ref{prop_convolution}). Since we proved that the sequence $(\phi_j)_{j\in \N}$ is bounded in $L^1_{\mathrm{loc}}(A)$, for fixed $z\in A_\epsilon$, $r\in (0,\epsilon/2)$ and $x,y\in \mathbb{B}_r(z)$ we deduce
\[
|\phi_j \star \chi_\epsilon(x)-\phi_j \star \chi_\epsilon(y) | \leq \| \phi_j \|_{L^1({\mathbb{B}_r(z))}}\sup_{|t|\leq \epsilon }|\chi(x-t)-\chi(y-t)| \,.
\]
In particular, by Ascoli-Arzelà Theorem $\phi_j \star \chi_\epsilon \to \phi \star \chi_\epsilon $ locally uniformly on $A$ as $j\to +\infty$. As a consequence, taking the $\limsup$ as $j\to +\infty$ of the inequality \eqref{inH} we get
\[
\limsup_{j\to+ \infty}\int_{K} |\phi_j-\phi| \, d \lambda  \leq 2 \int_A \eta (\phi \star \chi_\epsilon-\phi) \, d\lambda \,,
\]
where we also used that $\phi_j$ converges to $\phi$ weakly. Taking the limit as $\epsilon \searrow 0$ concludes the proof, as the monotone convergence theorem guarantees that $\int_A \eta (\phi \star \chi_\epsilon - \phi) \, d \lambda \to 0$.
\end{proof}


\subsection{Further technical results}
In this subsection we prove two technical local results used in the proof of Theorem \ref{Thm:continuity}. 

\begin{lem}\label{Cor:Hartogs}
Let $A\subseteq \H^n$ be a domain. Let $(\phi_j)_{j\in \N} \subset \QPSH(A)\cap C^0(A)$ be a uniformly bounded sequence that converges pointwise to a positive function $\phi\in \QPSH(A)$. If there exist $c>0$ and $t>1$ such that
\[
\phi-t\phi_*<c
\]
on a compact set $K\subseteq A$, then there exists an index $j_0\in \N $ and open neighbourhood $U\subseteq A$ of $K$ such that
\[
\phi_j-t\phi<c
\]
on $U$, for all $j\geq j_0$.
\end{lem}
As in \eqref{eqn:lsc_reg}, $\varphi_*$ denotes the lower semicontinuous regularisation of $\varphi$.
\begin{proof}
Take $x_0\in K$. Then there exists $0<c_1<c$ such that $
\phi(x_0)-t\phi_*(x_0)<c_1$. Since $\phi$ is upper semicontinuous there exists an open neighbourhood $V$ of $x_0$ such that
\begin{equation}\label{Hart1}
\phi(x)<t\phi_*(x_0)+c_1 \text{ for all }x\in V\,.
\end{equation}
On the other hand, by definition of lower semicontinuous regularisation, up to shrinking $V$, we may assume that there exists $c_1<c_2<c$ such that
\begin{equation}\label{Hart2}
t\phi_*(x_0)-c_2+c_1<t\inf_V \phi\,.
\end{equation}
Thus, combining \eqref{Hart1} and \eqref{Hart2} we obtain
\begin{equation}\label{Hart3}
\phi(x)<t\inf_V \phi + c_2 \text{ for all }x\in V\,.
\end{equation}
By Hartogs' Lemma for subharmonic functions (see e.g. \cite[Thm. 2.6.4]{Klimek}), for each $\epsilon >0$ and each compact neighbourhood $K_0 \subseteq V$ of $x_0$ there exists an index $j_0$ such that $\phi_j\leq \sup_V \phi +\epsilon $ on $K_0$ for any $j\geq j_0$. Choosing $\epsilon=c-c_2$ and recalling \eqref{Hart3} we obtain
\[
\phi_j(x)\leq \sup_V \phi +c-c_2<t\inf_V \phi + c \leq t\phi(x)+c
\]
for all $x\in K_0$ and $j\geq j_0$. Since $K$ is compact and the point $x_0\in K$ was chosen arbitrarily the inequality remains true on an open neighbourhood $U$ of $K$, up to increasing $j_0$.
\end{proof}

\begin{prop}\label{Prop:supCap}
Let $A\subseteq \H^n$ be a bounded domain with the standard HKT structure $\Omega_0$. Let $w \in \QPSH(A)\cap L^\infty(\bar A)$ and $\psi \in \QPSH(A)\cap C^0(\bar A)$ be such that
\[
\liminf_{A\ni x\to y}(w(x)-\psi(x)) \geq 0 \qquad \text{for any }y\in \partial A\,.
\]
If 
\[
(\partial \partial_J w)^n=h\,\Omega^n_0 \qquad \text{on }A\,,
\]
for some $h\in L^\infty(A)$, then for any $\alpha \in (0,1)$ there exists a constant $C>0$ such that
\[
\sup_A(\psi-w)\leq \epsilon +C \cdot \mathrm{Cap^{WK}}\bigl(\{\psi-w>\epsilon\},A\bigr)^{\alpha/n}\,,
\]
for any $\epsilon >0$.
\end{prop}
\begin{proof}
For $s\geq 0$ define
\[
f(s):=\mathrm{Cap^{WK}}\bigl(\{w-\psi<-s-\epsilon\}, A\bigr)^{1/n}\,.
\]
The function $f$ is right-continuous, decreasing and satisfies $\lim_{s\to +\infty}f(s)=0$ (because $w$ and $\psi$ are bounded). Now, for any $t\in [0,1]$, the following inequality holds
\begin{equation}\label{ineq}
t^nf(s+t)^n\leq \int_{\{w-\psi<-s-\epsilon\}} (\partial \partial_J w)^n\,.
\end{equation}
The proof is similar to that of Lemma \ref{Lem_CapMA}. Take $u\in \QPSH(A)$ such that $0\leq u\leq 1$ and note the inclusions
\[
\{w-\psi<-s-t-\epsilon\} \subseteq \left\{ w -\psi< tu-s-t-\epsilon \right\} \subseteq \{w-\psi <-s-\epsilon\}\,.
\]
Thanks to these, the fact that $(\partial \partial_J (tu))^n\leq(\partial \partial_J (\psi+tu-s-t-\epsilon))^n $ and the comparison principle of \cite[Lem. 2.5]{Wan-Kang} we deduce
\[
\begin{split}
t^n\int_{\{w-\psi<-s-t-\epsilon\}}(\partial \partial_J u)^n&\leq \int_{\{w-\psi<-s-t-\epsilon\}}(\partial \partial_J (\psi+tu-s-t-\epsilon))^n\\
&\leq\int_{\{ w -\psi< tu-s-t-\epsilon\}}(\partial \partial_J (\psi+tu-s-t-\epsilon))^n\\
&\leq\int_{\{ w -\psi< tu-s-t-\epsilon\}}(\partial \partial_J w)^n \leq \int_{\{ w -\psi<-s-\epsilon\}}(\partial \partial_J w)^n \,.
\end{split}
\]
Taking the supremum over all $u$'s yields \eqref{ineq}. Now, by assumption $(\partial \partial_J w)^n=h\, \Omega^n_0$ and thus, for any $\alpha \in (0,1)$, thanks to \cite[Lem. 9]{Sro}, we have
\[
\begin{split}
t^nf(s+t)^n&\leq \int_{\{w-\psi<-s-\epsilon\}} h\,\Omega^n_0\leq \|h\|_{L^\infty} \int_{\{w-\psi<-s-\epsilon\}}\Omega^n_0\\
&\leq C_\alpha \|h\|_{L^\infty} \mathrm{Cap^{WK}} \bigl(\{w-\psi<-s-\epsilon\},A\bigr)^{1+\alpha}=C_\alpha \|h\|_{L^\infty}f(s)^{n(1+\alpha)} \,,
\end{split}
\]
where $C_\alpha>0$ is a constant only depending on $A, \Omega_0$ and $\alpha$. Hence, $f$ satisfies \eqref{eq_boundedness} with $B=(C_\alpha \|h\|_{L^\infty})^{1/n}$. We can now apply Lemma \ref{Lem:EGZ} and deduce that there exists $S>0$ such that 
\[
\sup_A (\psi-w)\leq S+\epsilon\,.
\]
If  $f(0)^\alpha\geq (2B)^{-1}$ we deduce immediately the desired inequality:
\[
\sup_A (\psi-w)\leq S+\epsilon\leq 2B S f(0)^\alpha+\epsilon=2B S \mathrm{Cap^{WK}}\bigl(\{w-\psi<-\epsilon\},A\bigr)^{\alpha/n}+\epsilon\,.
\]
If instead $f(0)^\alpha<(2B)^{-1}$ we can take $s_0=0$ in \eqref{eqn:S} and conclude 
\[
\sup_A (\psi-w)\leq S+\epsilon\leq \frac{2B}{1-2^{-\alpha}}\mathrm{Cap^{WK}}\bigl(\{w-\psi<-\epsilon\},A\bigr)^{\alpha/n}+\epsilon\,. \qedhere
\]
\end{proof}


\begin{thebibliography}{[99]}
		
\bibitem{Alesker 2003}
{\sc S. Alesker}, Non-commutative linear algebra and plurisubharmonic functions of quaternionic variables, {\em Bull. Sci. Math.}, {\bf 127}(1), 1--35, 2003.

\bibitem{Alesker}
{\sc S. Alesker}, Quaternionic Monge-Ampère equations. {\em J. Geom. Anal.} {\bf 13}, no. 2, 205--238, 2003.
		
\bibitem{Alesker 2012}
{\sc S. Alesker}, Pluripotential theory on quaternionic manifolds, {\em J. Geom. Phys.}, {\bf 62}(5), 1189--1206, 2012.
		
\bibitem{Alesker (2013)}
{\sc S. Alesker}, Solvability of the quaternionic Monge-Amp\`{e}re equation on compact manifolds with a flat hyperK\"{a}hler metric, {\em Adv. Math.}, {\bf 241}, 192--219, 2013.
		
		
\bibitem{Alesker-Shelukhin (2017)}
{\sc S. Alesker, E. Shelukhin}, A uniform estimate for general quaternionic Calabi problem (with appendix by Daniel Barlet), {\em Adv. Math.}, {\bf 316}, 1--52, 2017.
		
\bibitem{Alesker-Verbitsky (2006)}
{\sc S. Alesker, M. Verbitsky}, Plurisubharmonic functions on hypercomplex manifolds and HKT-geometry, {\em J. Geom. Anal.}, {\bf 16}, 375--399, 2006.
		
\bibitem{Alesker-Verbitsky (2010)}
{\sc S. Alesker, M. Verbitsky}, Quaternionic Monge-Amp\`{e}re equations and Calabi problem for HKT-manifolds, {\em Israel J. Math.}, {\bf 176}, 109--138, 2010.


		
		
\bibitem{Banos}
{\sc B. Banos, A. Swann}, Potentials for Hyper-K\"ahler Metrics with Torsion, {\em Classical and Quantum Gravity}, {\bf 21}(13), 3127--3135, 2004.
		
\bibitem{Barberis-Fino}
{\sc M. L. Barberis, A. Fino}, New HKT manifolds arising from quaternionic representations, {\em Math. Z.}, {\bf 267}, 717--735, 2011.

\bibitem{Bedford-Taylor 1976}
{\sc E. Bedford, B. A. Taylor}, The Dirichlet problem for a complex Monge–Amp\`{e}re equation, {\em Invent. Math.}, {\bf 37}(1), 1--44, 1976.

\bibitem{Bedford-Taylor 1982}
{\sc E. Bedford, B. A. Taylor}, A new capacity for plurisubharmonic functions, {\em Acta Math.}, {\bf 149}, no. 1-2, 1--40, 1982.

\bibitem{BGV}
{\sc L. Bedulli, G. Gentili, L. Vezzoni}, The parabolic quaternionic Calabi-Yau equation on hyperkähler manifolds. {\em Rev. Mat. Iberoam.} {\bf 40} (2024), no. 6, 2291--2310.

\bibitem{Berman-Boucksom}
{\sc R. J. Berman, S. Boucksom}, Growth of balls of holomorphic sections and energy at equilibrium, {\em Invent. Math.}, {\bf 181} (2), 337--394, 2010.


\bibitem{Berman-Boucksom-Guedj-Zeriahi}
{\sc R. J. Berman, S. Boucksom, V. Guedj, A. Zeriahi}, A variational approach to complex Monge-Amp\`{e}re equations, {\em Publ. Math. Inst. Hautes \'{E}tudes Sci.} {\bf 117}, 179--245, 2013.

		
%
%
%

\bibitem{Blocki}
{\sc Z. B\l ocki}, Uniqueness and stability for the complex Monge-Ampère equation on compact K\"ahler manifolds, {\em Indiana Univ. Math. J.} {\bf 52} (2003), no. 6, 1697--1701.

\bibitem{Blocki-Kolodziej}
{\sc Z. B\l ocki and S. Ko\l odziej}, On regularization of plurisubharmonic functions on manifolds, {\em Proc. Amer. Math. Soc.} {\bf 135}, no. 7, 2089--2093, 2007.

%
		
\bibitem{Calabi}
{\sc E. Calabi}, On K\"ahler manifolds with vanishing canonical class, Algebraic geometry and topology. A symposium in honor of S. Lefschetz, 78--89. Princeton University Press, Princeton, N. J., 1957.
		

		


\bibitem{Cho-Choi}
{\sc Y.-W. Cho, Y.-J. Choi}, Continuity of solutions to complex Monge-Ampère equations on compact K\"ahler spaces. {\em Math. Ann.} {\bf 393}, no. 1, 807--830, 2025.

		
%
%
%

%
%

\bibitem{Dinew}
{\sc S. Dinew}, Uniqueness in $\mathscr{E}(X,\omega) $, {\em J. Funct. Anal.} {\bf 256} no. 7, 2113--2122, 2009.

\bibitem{DinewSroka}
{\sc S. Dinew, M. Sroka}, On the Alesker-Verbitsky conjecture on hyperK\"ahler manifolds, {\em Geom. Funct. Anal.} {\bf 33}, no. 4, 875--911, 2023.
		
\bibitem{dotti-fino1}
{\sc I. Dotti, A. Fino}, Abelian hypercomplex 8-dimensional nilmanifolds, {\em Ann. Glob. Anal. and Geom.} {\bf 18}, 47--59, 2000.
		
\bibitem{dotti-fino2}
{\sc I. Dotti, A. Fino}, Hyperk\"ahler torsion structures invariant by nilpotent Lie groups, {\em Classical Quantum Gravity} {\bf 19}, 551--562, 2002.

\bibitem{EGZ}
{\sc P. Eyssidieux, V. Guedj, A. Zeriahi}, Singular K\"ahler-Einstein metrics, {\em J. Amer. Math. Soc.} {\bf 22}, no. 3, 607--639, 2009.
		
%
%
		
\bibitem{GF}
{\sc A. Fino, G. Grantcharov}, Properties of manifolds with skew-symmetric torsion and special holonomy, {\em Adv. Math.} {\bf 189}, no. 2, 439--450, 2004.

\bibitem{GL}
{\sc G. Gentili, M. Lejmi}, On balanced HKT manifolds, {\em New York J. Math.} \textbf{31}, 1118--1139, 2025.


\bibitem{GentiliVezzoni}
{\sc G. Gentili, L. Vezzoni}, The quaternionic Calabi conjecture on abelian hypercomplex nilmanifolds viewed as tori fibrations. {\em Int. Math. Res. Not. IMRN} {\bf 2022}, no. 12, 9499--9528.
		
\bibitem{GV}
{\sc G. Gentili, L. Vezzoni}, A remark on the quaternionic Monge-Amp\`{e}re equation on foliated manifolds, {\em Proc. Amer. Math. Soc.} {\bf 151}, 1263--1275, 2023.

\bibitem{GV2}
{\sc G. Gentili, L. Vezzoni}, A remark on the second order estimates for the quaternionic Calabi-Yau problem on hyperk\"ahler manifolds, {\em Complex Manifolds} {\bf 13}, no. 1, article no. 20250020, 2026.

%

\bibitem{GLV}{\sc G. Grantcharov, M. Lejmi, M. Verbitsky},
Existence of HKT metrics on hypercomplex manifolds of real dimension 8, {\em Adv. Math.} {\bf 320}, 1135--1157, 2017.
		
\bibitem{Grantcharov-Poon (2000)}
{\sc G. Grantcharov, Y. S. Poon}, Geometry of hyperK\"{a}hler connections with torsion, {\em Comm. Math. Phys.} {\bf 213}(1), 19--37, 2000.
		


%
%
%
	
\bibitem{Guedj-Zeriahi 2005}
{\sc V. Guedj, A. Zeriahi}, Intrinsic capacities on compact K\"ahler manifolds, {\em J. Geom. Anal.}, {\bf 15}, no. 4, 607--639, 2005.

\bibitem{Guedj-Zeriahi 2007}
{\sc V. Guedj, A. Zeriahi}, The weighted Monge-Amp\`{e}re energy of quasiplurisubharmonic functions, {\em J. Funct. Anal.}, {\bf 250}, 442--482, 2007.

\bibitem{Guedj-Zeriahi book}
{\sc V. Guedj, A. Zeriahi}, Degenerate complex Monge-Ampère equations, EMS Tracts Math., 26 European Mathematical Society (EMS), Z\"urich, 2017.

%
%
%
%
		
\bibitem{Howe-Papadopoulos (1996)}
{\sc P. S. Howe, G. Papadopoulos}, Twistor spaces for hyper-K\"{a}hler manifolds with torsion, {\em Phys. Lett. B.} {\bf 379}, 80--86, 1996.
		
%
\bibitem{Ivanov}
{\sc S. Ivanov, A. Petkov}, HKT manifolds with holonomy ${\rm SL}(n,\H)$, {\em Int. Math. Res. Not. IMRN} {\bf 16}, 3779--3799, 2012.  
		


\bibitem{Kato}
{\sc M. Kato}, Compact differentiable $ 4 $-folds with quaternionic structures. {\em  Math. Ann.}, {\bf 248}, no. 1, 79--96, 1980. Erratum in {\em  Math. Ann.}, {\bf 283}, no. 2, 352, 1989.

\bibitem{Klimek}
{\sc M. Klimek}, Pluripotential theory.
{\em London Math. Soc. Monogr. (N.S.), 6
Oxford Sci. Publ.} The Clarendon Press, Oxford University Press, New York, 1991. xiv+266 pp.
		

\bibitem{K98}
{\sc S. Ko\l{}odziej}, The complex Monge-Amp\`{e}re equation, {\em Acta Math.}, {\bf 180}, no. 1, 69--117, 1998.

\bibitem{Kolodziej}
{\sc S. Ko\l{}odziej}, The Monge-Amp\`{e}re Equation on Compact K\"ahler Manifolds, {\em Indiana Univ. Math. J.}, {\bf 52}, no. 3, 667--686, 2003.

%
		
\bibitem{Lejmi-Weber}
{\sc M. Lejmi, P. Weber}, Quaternionic Bott–Chern Cohomology and existence of HKT metrics, {\em Q. J. Math.} {\bf 68} (3), 705--728, 2017.
		


\bibitem{Lu-Nguyen}
{\sc H. C. Lu, V. D. Nguyen}, Degenerate complex Hessian equations on compact K\"ahler manifolds, {\em Indiana Univ. Math. J.} {\bf 64}, 1721--1745, 2015.
%
%
		
\bibitem{Obata (1956)}
{\sc M. Obata}, Affine connections on manifolds with almost complex, quaternionic or Hermitian structures, {\em Japan. J. Math.} {\bf 26}, 43--79, 1956.
		
%


%
%




\bibitem{Sommese}
{\sc A. Sommese}, Quaternionic Manifolds, {\em Math. Ann.} {\bf 212}, 191--214, 1975.

		
\bibitem{Sro}
{ \sc M. Sroka}, Weak solutions to the quaternionic Monge-Amp\`ere equation, {\em Anal. PDE} {\bf 13}(6), 1755--1776, 2020.
		
\bibitem{Sroka}
{\sc M. Sroka}, The $ C^0 $ estimate for the quaternionic Calabi conjecture, {\em Adv. Math.} {\bf 370}, 107237, 2020.

\bibitem{Sroka24}
{\sc M. Sroka}, Sharp uniform bound for the quaternionic Monge–Ampère equation on hyperhermitian manifolds. {\em Calc. Var. Partial Differential Equations} {\bf63} (2024), no. 4, article no. 102, 14 pp.
		
%
		
\bibitem{Swann}
{\sc A. Swann}, Twisting Hermitian and hypercomplex geometries, {\em Duke Math. J.} {\bf 155}, no. 2, 403--431, 2010. 
%
%
%
%
%
%
%
%
		
\bibitem{Verbitsky (2002)}
{\sc M. Verbitsky}, HyperK\"{a}hler manifolds with torsion, supersymmetry and Hodge theory, {\em Asian J. Math.} {\bf 6}(4), 679--712, 2002.
		
\bibitem{Verbitsky (2007)}
{\sc M. Verbitsky},	Hypercomplex manifolds with trivial canonical bundle and their holonomy. {\em Moscow Seminar on Mathematical Physics.} II, 203--211, { \em  Amer. Math. Soc. Transl. Ser. 2}, {\bf 221}, Adv. Math. Sci., {\bf 60}, Amer. Math. Soc., Providence, RI, 2007.

\bibitem{Verbitsky (2009)}
{\sc M. Verbitsky}, Balanced HKT metrics and strong HKT metrics on hypercomplex manifolds, {\em Math. Res. Lett.} {\bf 16}, no. 4, 735--752, 2009.

\bibitem{Verbitsky (2010)}
{\sc M. Verbitsky}, Positive forms on hyperk\"ahler manifolds, {\em Osaka J. Math.} {\bf 47}, no. 2, 353--384, 2010.

\bibitem{Wan 2017}
{\sc D. Wan}, The continuity and range of the quaternionic Monge-Amp\`{e}re operator on quaternionic space, {\em Math. Z.} {\bf 285} no. 1-2, 461--478, 2017. 

\bibitem{Wan 2019}
{\sc D. Wan}, Quaternionic Monge-Amp\`{e}re operator for unbounded plurisubharmonic functions, {\em Ann. Mat. Pura Appl.} (4) {\bf 198}, no. 2, 381--398, 2019.

\bibitem{Wan 2019b}
{\sc D. Wan}, The domain of definition of the quaternionic Monge-Amp\`{e}re operator, {\em Math. Nachr.} {\bf 292} no. 5, 1161--1173, 2019. 

\bibitem{Wan 2020}
{\sc D. Wan}, A variational approach to the quaternionic Monge-Amp\`{e}re equation, {\em Ann. Mat. Pura Appl.} (4) {\bf 199}, no. 6, 2125--2150, 2020.

\bibitem{Wan-Kang}
{\sc D. Wan, Q. Kang}, Potential theory for quaternionic plurisubharmonic functions, {\em Mich. Math. J.} {\bf 66}, 3--20, 2017.

\bibitem{Wan-Wang}
{\sc D. Wan, W. Wang}, On the quaternionic Monge-Amp\`{e}re operator, closed positive currents and Lelong-Jensen type formula on the quaternionic space, {\em Bull. Sci. Math.}, {\bf 141}, 267--311, 2017.

\bibitem{Wan-Zhang}
{\sc D. Wan, W. Zhang}, Quasicontinuity and maximality of quaternionic plurisubharmonic functions, {\em J. Math. Anal. Appl.} {\bf 424}, no. 1, 86--103, 2015.

\bibitem{Wang (2021)}
{\sc W. Wang}, The quaternionic Monge-Amp\`{e}re operator and plurisubharmonic functions on the Heisenberg group, {\em Math. Z.} {\bf 298}, no. 1-2, 521--549, 2021.
		
\bibitem{Yau}
 {\sc S.-T. Yau}, On the Ricci curvature of a compact K\"ahler manifold and the complex Monge-Amp\`ere equation. I. {\em Comm. Pure Appl. Math.} {\bf 31}, no. 3, 339--411, 1978.
		

		
		
\end{thebibliography}
\end{document}